\documentclass[11pt,a4paper]{article}
\usepackage{a4wide}
\usepackage{amsfonts}
\usepackage{amsmath,latexsym,amssymb,amsfonts,amsbsy, amsthm}
\usepackage[hidelinks]{hyperref}
\usepackage{cite}
\usepackage[usenames]{color}
\usepackage{xspace,colortbl}
\allowdisplaybreaks

\newcommand{\beq}{\begin{equation}}
\newcommand{\eeq}{\end{equation}}
\newcommand{\ben}{\begin{eqnarray}}
\newcommand{\een}{\end{eqnarray}}
\newcommand{\beno}{\begin{eqnarray*}}
\newcommand{\eeno}{\end{eqnarray*}}

\theoremstyle{plain}
\newtheorem{theorem}{Theorem}[section]

\newtheorem{corollary}[theorem]{Corollary}
\newtheorem{proposition}[theorem]{Proposition}
\newtheorem{lemma}[theorem]{Lemma}
\newtheorem{remark}[theorem]{Remark}

\numberwithin{theorem}{section} \numberwithin{equation}{section}
\renewcommand{\theequation}{\thesection.\arabic{equation}}

\newcommand{\average}{{\mathchoice {\kern1ex\vcenter{\hrule height.4pt
width 6pt depth0pt} \kern-9.7pt} {\kern1ex\vcenter{\hrule height.4pt
width 4.3pt depth0pt} \kern-7pt} {} {} }}

\def\R{\mathbb{R}}

\newcommand{\del}{\partial }

\newcommand{\sdiamond}{\text{\tiny $\diamond$}}

\renewcommand{\phi}{\varphi}

\newcommand{\ov}{\overline}

\newcommand{\be}{\begin{equation}}
\newcommand{\ee}{\end{equation}}

\DeclareMathOperator{\sgn}{sgn}

\newcommand{\N}{\mathbb{N}}

\newcommand{\cM}{{\mathcal M}}

\newcommand{\cU}{{\mathcal U}}

\newcommand{\B}{{\bf B}}
\renewcommand{\dim}{{\rm dim}\,}
\newcommand{\eps}{\varepsilon}

\renewcommand{\theequation}{\thesection.\arabic{equation}}

\renewcommand{\epsilon}{\varepsilon}

\begin{document}
\title{Spectral asymptotics of radial solutions and nonradial bifurcation for the H{\'e}non equation}
\author{Joel K\"ubler\thanks{Institut f\"ur Mathematik, Goethe-Universit\"at Frankfurt, Robert-Mayer-Str. 10,
		D-60629 Frankfurt a.M., kuebler@math.uni-frankfurt.de.}
\and 
Tobias Weth\thanks{Institut f\"ur Mathematik, Goethe-Universit\"at Frankfurt, Robert-Mayer-Str. 10,
		D-60629 Frankfurt a.M., weth@math.uni-frankfurt.de.}}
\maketitle
\begin{abstract}
	We study the spectral asymptotics of nodal (i.e., sign-changing) solutions of the problem
	\begin{equation*}
(H) \qquad \qquad		\left \{
		\begin{aligned}
			-\Delta u &=|x|^\alpha |u|^{p-2}u&&\qquad \text{in $\B$,}\\
			u&=0&&\qquad \text{on $\partial \B$,}
		\end{aligned}
		\right.
	\end{equation*}
	in the unit ball $\B \subset \R^N,N\geq 3$, $p>2$ in the limit $\alpha \to +\infty$. More precisely, for a given positive integer $K$, we derive asymptotic $C^1$-expansions for the negative eigenvalues of the linearization of the unique radial solution $u_\alpha$ of $(H)$ with precisely $K$ nodal domains and $u_\alpha(0)>0$. As an application, we derive the existence of an unbounded sequence of bifurcation points on the radial solution branch $\alpha \mapsto (\alpha,u_\alpha)$ which all give rise to bifurcation of nonradial solutions whose nodal sets remain homeomorphic to a disjoint union of concentric spheres. 
\end{abstract}

\vskip 0.2cm


\renewcommand{\theequation}{\thesection.\arabic{equation}}
\setcounter{equation}{0}
\section{Introduction}
We consider the Dirichlet problem for the generalized H\'{e}non equation
\begin{equation}\label{1.4}
\left \{
 \begin{aligned}
-\Delta u &=|x|^\alpha |u|^{p-2}u&&\qquad \text{in $\B$,}\\
u&=0&&\qquad \text{on $\partial \B$,}
\end{aligned}
\right.
\end{equation}
where $\B \subset \R^N,N\geq 3$ is the unit ball and $p>2$, $\alpha>0$.
This equation originally arose through the study of stellar clusters in \cite{Henon}. 
One of the first results on (\ref{1.4}) is due to Ni \cite{Ni}, who proved the existence of a positive radial solution in the subcritical range of exponents $2<p<2_\alpha^*$, where $2_\alpha^*:= \frac{2N+2\alpha}{N-2}$. In another seminal paper, Smets, Willem and Su  \cite{Smets-Willem-Su} observed that symmetry breaking occurs for fixed $p$ and large $\alpha$, i.e., there exists $\alpha^*>0$ depending on $p$ such that ground state solutions of \eqref{1.4} are nonradial for $\alpha>\alpha^*$. In the sequel, the existence and shape of radial and nonradial solutions of the H\'enon equation has received extensive attention, see e.g. \cite{Smets-Willem, Cao-Peng, Serra, Pistoia-Serra, B-W,B-W2, Amadori-Gladiali,Amadori-Gladiali2,Amadori-Gladiali3,lou-weth-zhang:2018}.
In particular, bifurcation of nonradial positive solutions in the parameter $p$ is studied in \cite{Amadori-Gladiali} for fixed $\alpha>0$. Moreover, a related critical parameter-dependent equation on $\R^N$ is considered in \cite{Gladiali-Grossi-Neves}. 

The main motivation for the present paper is the investigation of bifurcation of nonradial nodal (i.e., sign changing) solutions -- in the parameter $\alpha>0$ -- from the set of radial nodal solutions. To explain this in more detail, let us fix $K \in \N$, an exponent $p>2$ and consider 
$$
\alpha > \alpha_p:= \max \left\{\frac{(N-2)p - 2N}{2},0 \right\}, 
$$ 
which amounts to the subcriticality condition $p<2_\alpha^*$. Under these assumptions, it has been proved by Nagasaki \cite{nagasaki} that (\ref{1.4}) admits a unique classical radial solution $u_\alpha \in C^2(\overline \B)$ with $u_\alpha(0)>0$ and with precisely $K$ nodal domains (i.e., $K-1$ zeros in the radial variable $r= |x| \in (0,1)$). In order to decide whether the branch $\alpha \to u_\alpha$ admits bifurcation of nonradial solutions for large $\alpha$, we need to analyze its spectral asymptotics as $\alpha \to \infty$. More precisely, we wish to derive asymptotic expansions of the eigenvalues of the linearizations of (\ref{1.4}) at $u_\alpha$ as $\alpha \to \infty$. For this we consider the linearized operators 
\begin{equation} \label{linearized operator}
\phi \mapsto L^\alpha \phi:= -\Delta \phi - (p-1) |x|^\alpha |u_\alpha|^{p-2} \phi, \qquad \alpha> \alpha_p,
\end{equation}
which are self-adjoint operators in $L^2(\B)$ with compact resolvent, domain 
$H^2(\B) \cap H^1_0(\B)$ and form domain $H^1_0(\B)$. In particular, they are Fredholm operators of index zero.

As usual, $u_\alpha$ is called nondegenerate if $L^\alpha: H^2(\B) \cap H^1_0(\B) \to L^2(\B)$ is an isomorphism, which amounts to the property that the equation $L^\alpha \phi = 0$ only has the trivial solution $\phi=0$ in $H^2(\B) \cap H^1_0(\B)$. Otherwise, $u_\alpha$ is called degenerate. By a classical observation, only values $\alpha$ such that $u_\alpha$ is degenerate can give rise to bifurcation from 
the branch $\alpha \mapsto u_\alpha$. Moreover, properties of the kernel of $L^\alpha$ and the change of the Morse index are of key importance to establish bifurcation. Here we recall that the Morse index of $u_\alpha$ is defined as the number of negative eigenvalues of the operator $L^\alpha$. 

The first step in deriving asymptotic spectral information of the operator family $L^\alpha$, $\alpha>\alpha_p$ is to characterize the limit shape of the solutions $u_\alpha$ after suitable transformations. Inspired by Byeon and Wang \cite{B-W}, we transform the radial variable and derive a corresponding limit problem. Here, for simplicity, we also regard $u_\alpha = u_\alpha(r)$ as a function of the radial variable $r=|x| \in [0,1]$. Our first preliminary result is the following. 

\begin{proposition}
\label{limit-shape} Let $p>2$, $K \in \N$. Moreover, for $\alpha>\alpha_p$, let $u_\alpha$ denote the unique radial solution of (\ref{1.4}) with $K$ nodal domains and $u_\alpha(0)>0$, and define
  \begin{equation}
    \label{eq:U-alpha-definition}
U_\alpha: [0,\infty) \to \R, \qquad U_\alpha(t)= (N+\alpha)^{-\frac{2}{p-2}}\:
u_\alpha(e^{-\frac{t}{N+\alpha}}).
  \end{equation}
Then $U_\alpha \to (-1)^{K-1} U_\infty$ uniformly on $[0,\infty)$ as $\alpha \to \infty$, where $U_\infty \in C^2([0,\infty))$  is characterized as the unique bounded solution of the limit problem 
 \begin{equation}
  \label{eq:limit-U}
-U'' = e^{-t}|U|^{p-2}U \quad \text{in $[0,\infty)$,}\qquad U(0)=0  
\end{equation}
with $U'(0)>0$ and with precisely $K-1$ zeros in $(0,\infty)$. 
\end{proposition}

The asymptotic description derived in Proposition~\ref{limit-shape} implies that the solutions $u_\alpha$ blow up everywhere in $\B$ as $\alpha \to \infty$, in contrast to the nonradial ground states considered in \cite{Smets-Willem-Su}. It is therefore reasonable to expect that the Morse index of $u_\alpha$ tends to infinity as $\alpha \to \infty$. This fact has been proved recently and independently for more general classes of problems in \cite{Amadori-Gladiali2,lou-weth-zhang:2018}, extending a result for the case $N=2$ given in \cite{Moreira-dos-Santos-Pacella}. To obtain a more precise description of the distribution of eigenvalues of $L^\alpha$ as $\alpha \to \infty$, we rely on complementary approaches of \cite{Amadori-Gladiali2,lou-weth-zhang:2018} and implement new tools. We note here that \cite{lou-weth-zhang:2018} uses the transformation~(\ref{eq:U-alpha-definition}) in a more general context together with Liouville type theorems for limiting problems on the half line. In the present paper, we build on very useful results obtained recently by Amadori and Gladiali in \cite{Amadori-Gladiali2}. In particular, we use the fact that the Morse index of $u_\alpha$ equals the number of negative eigenvalues (counted with multiplicity) of the weighted eigenvalue problem 
\begin{equation}
  \label{eq:weighted-L-alpha-problem}
L^\alpha \phi =\frac{\lambda}{|x|^2}\phi,\qquad \phi \in H^1_0(\B), 
\end{equation}
see \cite[Prop. 5.1]{Amadori-Gladiali2}. In various special cases, this observation had already been used before, see e.g. \cite[Section 5]{Dancer-Gladiali-Grossi}. In order to avoid regularity issues related to the singularity of the weight $\frac{1}{|x|^2}$, it is convenient to consider (\ref{eq:weighted-L-alpha-problem}) in weak sense via the quadratic form $q_\alpha$ associated with $L^\alpha$, see Section~\ref{sec:spectr-asympt-refeq} below.  The problem (\ref{eq:weighted-L-alpha-problem}) is easier to analyze than the standard eigenvalue problem 
$L^\alpha \phi =\lambda \phi$ without weight. Indeed, every eigenfunction of (\ref{eq:weighted-L-alpha-problem}) is a sum of functions of the form 
\begin{equation}
  \label{eq:eigenfunction-form}
x \mapsto \phi(x) = \psi(x)Y_{\ell}\left(\frac{x}{|x|}\right), 
\end{equation}
where $\psi \in H^1_{0,rad}(\B)$ and $Y_\ell$ is a spherical harmonic of degree $\ell$, see \cite[Prop. 4.1]{Amadori-Gladiali2}. Here $H^1_{0,rad}(\B)$ denotes the space of radial functions in $H^1_0(\B)$. We recall that the space of spherical harmonics of degree $\ell \in \N \cup \{0\}$ has dimension 
$d_\ell := {N+ \ell -1 \choose N-1} - {N+ \ell -3 \choose N-1}$, and that 
every such spherical harmonic is an eigenfunction of the Laplace-Beltrami operator on the
unit sphere $\mathbb{S}^{N-1}$ corresponding to the eigenvalue $\lambda_\ell:= \ell (\ell + N-2)$. For functions $\phi$ of the form (\ref{eq:eigenfunction-form}), the eigenvalue problem (\ref{eq:weighted-L-alpha-problem}) reduces to an eigenvalue problem for radial functions given by 
\begin{equation}
\label{eq:weighted-eigenvalue-problem-reduced}
L^\alpha \psi =\frac{\mu}{|x|^2}\psi,\qquad \psi \in H^1_{0,rad}(\B),
\end{equation}
where $\mu = \lambda - \lambda_{\ell}$. In \cite[p.19 and Prop. 3.7]{Amadori-Gladiali2}, it has been proved that (\ref{eq:weighted-eigenvalue-problem-reduced}) admits precisely $K$ negative eigenvalues 
\begin{equation}
\label{eigenvalue-curves-first}
\mu_1(\alpha) <\mu_2(\alpha)<\dots< \mu_K(\alpha)<0 \qquad \qquad \text{for $\alpha> \alpha_p$.}
\end{equation}
Combining this fact with the observations summarized above, one may then derive the following facts which we cite here in a slightly modified form from \cite{Amadori-Gladiali2}.

\begin{proposition} (see \cite[Prop. 1.3 and 1.4]{Amadori-Gladiali2})\\
\label{spectral-curves-0}
Let $p > 2$ and $\alpha >\alpha_p$. Then the Morse index of $u_\alpha$ is given by 
$$
m(u_\alpha)= 
\sum \limits_{(i,\ell) \in E^-} d_\ell,
$$ 
where $E^-$ denotes the set of pairs $(i,\ell)$ with $i \in \N,\: \ell \in \N \cup \{0\}$ and $\mu_{i}(\alpha)+ \lambda_\ell <0$. Moreover, $u_\alpha$ is nondegenerate if and only if 
$$
\mu_i(\alpha) + \lambda_\ell \not = 0\qquad \text{for every $i \in \{1,\dots,K\}$, $\ell \in \N \cup \{0\}$.}
$$
\end{proposition}

In order to describe the asymptotic distribution of negative eigenvalues of $L^\alpha$, it is essential to study the asymptotics of the eigenvalues $\alpha \mapsto \mu_i(\alpha)$, $i=1,\dots,K$. With regard to this aspect, we mention the estimate 
\begin{equation}
\label{gladiali-mu-i-est}
\mu_i(\alpha) < - \frac{(\alpha+2)\bigl(\alpha+ 2(N-1)\bigr)}{4} \qquad \text{for $\alpha > \alpha_p$, $i = 1,\dots,K-1$,}
\end{equation}
which has been derived in \cite[Lemma 5.11 and Remark 5.12]{Amadori-Gladiali2}. In particular, it follows that $\mu_i(\alpha) \to -\infty$ as $\alpha \to \infty$ for $i = 1,\dots,K-1$. In our first main result, we complement this estimate by deriving asymptotics for $\mu_i(\alpha)$. 

\begin{theorem}
\label{spectral-curves}
Let $p > 2$ and $\alpha >\alpha_p$. Then the negative eigenvalues of (\ref{eq:weighted-eigenvalue-problem-reduced}) are given as $C^1$-functions 
$(\alpha_p,\infty) \to \R$, $\alpha \mapsto \mu_{i}(\alpha)$, $i = 1,\ldots,K$ 
satisfying the asymptotic expansions
\begin{equation} \label{expansions}
\mu_i(\alpha) =  \nu^*_i \alpha^2 +  c^*_i \alpha  +o(\alpha) \quad \text{and}\quad 
\mu_i'(\alpha) = 2 \nu^*_i \alpha  + c^*_i +o(1) \qquad \text{as $\alpha \to \infty$,}
\end{equation}
where $c^*_i$, $i=1,\dots,K$ are constants and the values $\nu^*_1 < \nu^*_2 < \dots< \nu^*_K < 0$ are precisely the negative eigenvalues of the eigenvalue problem   
\begin{equation}
  \label{eq:weighted-eigenvalue-translimit-preliminaries}
\left\{
  \begin{aligned}
&-\Psi'' - (p-1)e^{-t}|U_\infty(t)|^{p-2}\Psi= 
\nu \Psi \quad \text{in $[0,\infty)$,}\\
&\qquad \Psi(0)=0,\quad \Psi \in L^\infty(0,\infty),
\end{aligned}
\right.
\end{equation}
with $U_\infty$ given in Proposition~\ref{limit-shape}. In particular, there exists $\alpha^*>0$ such that the curves $\mu_i$, $i=1,\dots,K$ are strictly decreasing on $[\alpha^*, \infty)$.
\end{theorem}

\begin{remark}\label{introduction-remark-1}
{\rm The strict monotonicity of the curves $\mu_i$ on $[\alpha^*, \infty)$ will be of key importance for the derivation of bifurcation of nonradial solutions via variational bifurcation theory. For this we require the derivative expansion in (\ref{expansions}), but we do not need additional information on the constants $c_i^*$ since $\nu^*_i<0$ for $i=1,\dots,K$. Our proof of (\ref{expansions}) gives rise to the following characterization of the constants $c_i^*$: For fixed $i \in \{1,\dots,K\}$, we have 
$$
c_i^*=  -(2N \nu_i^* + N-2)(p-1) \int_0^\infty \left( t e^{-t} |U_\infty|^{p-2}\Psi^2 + (p-2) e^{-t} |U_\infty|^{p-4} U_\infty 
V \Psi^2 \right) \, dt,
$$
where $U_\infty$ is given in Proposition~\ref{limit-shape}, $V$ is the unique bounded solution of the problem   
$$
-V'' - (p-1)e^{-t}|U_\infty|^{p-2}V = U_\infty' - t e^{-t}|U_\infty|^{p-2}U_\infty \quad \text{in $[0,\infty)$,}\qquad V(0)=0
$$
and $\Psi$ is the (up to sign unique) eigenfunction of (\ref{eq:weighted-eigenvalue-translimit-preliminaries}) associated with the eigenvalue $\nu_i^*$ with $\int_0^\infty \Psi^2 \, dt =1$.}  
\end{remark}

The strict monotonicity of the curves $\mu_i$ for large $\alpha$ asserted in Theorem \ref{spectral-curves} allows us to deduce the following useful properties related to nondegeneracy and a change of the Morse index of the functions $u_\alpha$.

\begin{corollary}
\label{corollary-on-eigenvalue-curves}    
Let $p > 2$. For every $i \in \{1,\dots,K\}$, there exist $\ell_i \in \N \cup \{0\}$ and sequences of numbers $\alpha_{i,\ell} \in (\alpha_p,\infty)$, $\eps_{i,\ell}>0$, $\ell \ge \ell_i$ with the following properties:
\begin{itemize}
\item[(i)] $\alpha_{i,\ell} \to \infty$ as $\ell \to \infty$.
\item[(ii)]  $\mu_i(\alpha_{i,\ell})+ \lambda_\ell = 0$. In particular, $u_{\alpha_{i,\ell}}$ is degenerate.  
\item[(iii)] $u_\alpha$ is nondegenerate for $\alpha \in (\alpha_{i,\ell}-\eps_{i,\ell},\alpha_{i,\ell}+\eps_{i,\ell})$, $\alpha \not = \alpha_{i,\ell}$.
\item[(iv)] For $\eps \in (0,\eps_{i,\ell})$ the Morse index of $u_{\alpha_{i,l}+\eps}$ is strictly larger than the Morse index of $u_{\alpha_{i,l}-\eps}$.
\end{itemize}
\end{corollary}

With the help of Corollary~\ref{corollary-on-eigenvalue-curves} and an abstract bifurcation result in \cite{kielhoefer:1988}, we will derive our second main result on the bifurcation of nonradial solutions from the branch $\alpha \mapsto u_\alpha$.

\begin{theorem}
\label{thm-bifurcation}
Let $2<p<\frac{2N}{N-2}$, and let $K \in \N$, $i \in \{1,\dots,K\}$ be fixed. 
Then the points $\alpha_{i,\ell}$, $\ell \ge \ell_i$ are bifurcation points for nonradial solutions of (\ref{1.4}).

 More precisely, for every $\ell \ge \ell_i$, there exists a sequence $(\alpha_n,u^n)_n$ in $(0,\infty) \times C^2(\ov\B)$ with the following properties: 
\begin{itemize}
\item[(i)] $\alpha_n \to \alpha_{i,\ell}$, and $u^n \to u_{\alpha_{i,\ell}}$ in $C^2(\ov\B)$.
\item[(ii)] For every $n \in \N$, $u^n$ is a nonradial solution of (\ref{1.4}) with $\alpha= \alpha_n$ having precisely $K$ nodal domains $\Omega_1,\dots,\Omega_K$ such that $0 \in \Omega_1$, $\Omega_1$ is homeomorphic to a ball and $\Omega_2,\dots,\Omega_K$ are homeomorphic to annuli. 
\end{itemize}
Here, $\ell_i \in \N \cup \{0\}$ and the values $\alpha_{i,\ell}$ are given in Corollary~\ref{corollary-on-eigenvalue-curves}.
\end{theorem}

As mentioned above, Theorem~\ref{thm-bifurcation} will be derived from Corollary~\ref{corollary-on-eigenvalue-curves} and variational bifurcation theory. For this we reformulate (\ref{1.4}) as a bifurcation equation in the Hilbert space $H^1_0(\B)$ and show that, as a consequence of Corollary~\ref{corollary-on-eigenvalue-curves}, the crossing number of an associated operator family is nonzero at the points $\alpha_{i,\ell}$. Thus the main theorem in \cite{kielhoefer:1988} applies and yields that the points $\alpha_{i,\ell}$, $\ell \ge \ell_i$ are bifurcation points for solutions of (\ref{1.4}) along the branch $\alpha \mapsto u_\alpha$. To see that bifurcation of {\em nonradial} solutions occurs, it suffices to note that the solutions $u_\alpha$ are radially nondegenerate for $\alpha>0$, i.e., the kernel of $L^\alpha$ does not contain radial functions. A proof of the latter fact can be found in \cite[Theorem 1.7]{Amadori-Gladiali2}, and it also follows from results in \cite{Yanagida}. 

Since Corollary~\ref{corollary-on-eigenvalue-curves} is a rather direct consequence of Theorem~\ref{spectral-curves}, the major part of this paper is concerned with the proofs of Proposition~\ref{limit-shape} and Theorem~\ref{spectral-curves}. It is not difficult to see that, via the transformation given in (\ref{eq:U-alpha-definition}), the H\'enon equation (\ref{1.4}) transforms into a family of problems depending on the new parameter $\gamma=\frac{N-2}{N+\alpha}$ which admits a well-defined limit problem as $\gamma \to 0^+$ given by (\ref{eq:limit-U}). It is then necessary to choose a proper function space which allows to apply the implicit function theorem at $\gamma=0$, and this yields the convergence statement in Proposition~\ref{limit-shape}. The idea of the proof of Theorem~\ref{spectral-curves} is similar, as we use the same transformation (up to scaling) to rewrite the $\alpha$-dependent eigenvalue problem (\ref{eq:weighted-eigenvalue-problem-reduced}) as a $\gamma$-dependent eigenvalue problem on the interval $[0,\infty)$. We shall then see that (\ref{eq:weighted-eigenvalue-translimit-preliminaries}) arises as the limit of the transformed eigenvalue problems as $\gamma \to 0^+$. In order to obtain $C^1$-expansions of eigenvalue curves, we wish to apply the implicit function theorem again at the point $\gamma=0$. Here a major difficulty arises in the case where $p \in (2,3]$, as the map $U \mapsto |U|^{p-2}$ fails to be differentiable between standard function spaces. We overcome this problem by restricting this map to the subset of $C^1$-functions on $[0,\infty)$ having only a finite number of simple zeros and by considering its differentiability with respect to a weighted uniform 
$L^1$-norm, see Sections~\ref{sec:spectr-asympt-refeq} and \ref{sec:differentiability-g}. This is certainly the hardest step in the proof of Theorem~\ref{spectral-curves}.

It seems instructive to compare the transformations used in the present paper with the ones used in \cite{Moreira-dos-Santos-Pacella,Amadori-Gladiali2}. Transforming a radial solution $u$ of (\ref{1.4}) by setting 
$w(\tau)=(\frac{2}{2+\alpha})^{\frac{2}{p-2}}u(\tau^{\frac{2}{2+\alpha}})$ for $\tau \in (0,1)$ leads to the problem 
\begin{equation}
\label{gladiali-transform}
-(t^{M-1}w')'=t^{M-1}|w|^{p-2}w\quad \text{in $(0,1)$,}\qquad \quad w'(0)= w(1)=0
\end{equation}
with $M=M(\alpha)= \frac{2(N+\alpha)}{2+\alpha}$. Via this transformation, the associated weighted singular eigenvalue problem (\ref{eq:weighted-eigenvalue-problem-reduced}) corresponds to the even more singular eigenvalue equation
\begin{equation}
\label{gladiali-transform-eigenvalue}
-(t^{M-1}\psi')'-(p-1) t^{M-1}|w|^{p-2}\psi= t^{M-3}\hat \nu \psi \qquad \text{in $(0,1)$},
\end{equation}
which is considered in $M$-dependent function spaces in \cite{Amadori-Gladiali2}. In principle, it should be possible to carry out our approach also via these transformations, but we found it easier to find appropriate parameter-independent function spaces in the framework we use here. We stress again that finding parameter-independent function spaces is essential for the application of the implicit function theorem. 

The paper is organized as follows.  
In Section~\ref{sec:limit-alpha-to}, we first recall some known results on radial solutions of \eqref{1.4} and properties of the associated linearized operators. We then study the asymptotic behavior of the functions $u_\alpha$ as $\alpha \to \infty$ and prove Proposition~\ref{limit-shape}.
Section~\ref{sec:spectr-asympt-refeq} is devoted to the proofs of Theorem~\ref{spectral-curves} and Corollary~\ref{corollary-on-eigenvalue-curves}.
In Section~\ref{sec:differentiability-g} we prove, in particular, the differentiability of the map $U \mapsto |U|^{p-2}$ for $p \in (2,3]$ in a suitable functional setting. 
In Section~\ref{sec:bifurc-almost-radi}, we finally prove the bifurcation result stated in Theorem~\ref{thm-bifurcation}.

\subsection*{Acknowledgement}

The authors wish to thank Francesca Gladiali for helpful discussions and for pointing out the paper \cite{Amadori-Gladiali2}.

\section{The limit shape of sign changing radial solutions of (\ref{1.4}) as $\alpha \to \infty$}
\label{sec:limit-alpha-to}

This section is devoted to the asymptotics of branches of sign changing radial solutions of (\ref{1.4}) as $\alpha \to \infty$. In particular, we will prove Proposition~\ref{limit-shape}. As before, we let $K \in \N$ be fixed, and we first recall a result on the existence, uniqueness and radial Morse index of a radial solution $u_\alpha$ of (\ref{1.4}) with $K$ nodal domains.
\begin{theorem}
\label{sec:bifurc-nonr-solut-1}
For every $p>2$ and $\alpha >\alpha_p$, equation~(\ref{1.4}) has a unique radial solution $u_\alpha \in C^2(\overline \B)$ with precisely $K$ nodal domains such that $u_\alpha(0)>0$. Furthermore, the linearized operator 
$$
L^\alpha : H^2(\B) \cap H^1_0(\B) \to L^2(\B),\qquad  L^\alpha \phi:= -\Delta \phi - (p-1) |x|^\alpha |u_\alpha|^{p-2} \phi
$$
is a Fredholm operator of index zero having the following properties for every $\alpha  \ge 0$:
\begin{enumerate}
\item[(i)] $u_\alpha$ is radially nondegenerate in the sense that the kernel of $L^\alpha$ does not contain radial functions.
\item[(ii)] $u_\alpha$ has radial Morse index $K$ in the sense that $L^\alpha$ has precisely $K$ negative eigenvalues  
corresponding to radial eigenfunctions in $H^2(\B) \cap H^1_0(\B)$.
\end{enumerate}
\end{theorem}

Theorem~\ref{sec:bifurc-nonr-solut-1} is 
merely a combination of results in \cite{nagasaki} and \cite{Amadori-Gladiali2}. More precisely, the existence and uniqueness of $u_\alpha$ is proved in \cite{nagasaki}.
Note that the operator $L^\alpha$ is a compact perturbation of the isomorphism $-\Delta: H^2(\B) \cap H^1_0(\B) \to L^2(\B)$, which implies that it is a Fredholm operator of index zero. 
A proof of the radial nondegeneracy and radial Morse index can be found in \cite[Theorem 1.7]{Amadori-Gladiali2}. We remark here that the radial nondegeneracy can also be deduced from results in \cite{Yanagida}. 

\begin{remark}
(i) Since equation~(\ref{1.4}) remains invariant under a change of sign $u \mapsto -u$, it follows from Theorem~\ref{sec:bifurc-nonr-solut-1} that for every $p>2$ and $\alpha >\alpha_p$, equation~(\ref{1.4}) has precisely two radial solution $\pm u_\alpha \in C^2(\overline \B)$ with precisely $K$ nodal domains.\\
(ii) In \cite{nagasaki} it is also shown that for $p \geq \frac{2N + 2 \alpha}{N-2}$, the trivial solution is the only radial solution of equation~\eqref{1.4}.
\end{remark}

Next we recall that, in the radial variable, $u_\alpha$ solves 
\begin{equation}
  \label{eq:u-alpha-radial-variable-1}
-u_{rr} - \frac{N-1}{r}u_r = r^\alpha |u|^{p-2}u, \quad r \in (0,1), \qquad u'(0)=u(1)=0.  
\end{equation}
Inspired by Byeon-Wang \cite{B-W}, we transform equation (\ref{eq:u-alpha-radial-variable-1}), considering  
$$
U_\alpha: [0,\infty) \to \R, \qquad U_\alpha(t)= (N+\alpha)^{-\frac{2}{p-2}}\:
u_\alpha(e^{-\frac{t}{N+\alpha}}).
$$
By direct computation, we see that $U_\alpha$ is a bounded solution of the problem 
\begin{equation}
  \label{eq:U-gamma}
-(e^{-\gamma t}U')' = e^{-t}|U|^{p-2}U \quad \text{in $I:=[0,\infty)$,}\qquad U(0)=0.  
\end{equation}
with $\gamma= \gamma(\alpha)=\frac{N-2}{N+\alpha}$. Moreover, $U_\alpha$ has precisely $K-1$ zeros in $(0,\infty)$ and satisfies $\lim \limits_{t \to \infty}U_\alpha(t)>0$, which implies that $(-1)^{K-1}U_\alpha'(0)>0$. Considering the limit $\alpha \to \infty$ in (\ref{eq:u-alpha-radial-variable-1}) corresponds to sending $\gamma \to 0$ in (\ref{eq:U-gamma}), which leads to limit problem
\begin{equation}
\label{eq:limit-U-0-1}
-U'' = e^{-t}|U|^{p-2}U \quad \text{in $I$,}\qquad U(0)=0.  
\end{equation}
We first note the following facts regarding (\ref{eq:limit-U-0-1}).

\begin{proposition}
\label{sec:preliminaries-limit-problem-1}
Let $p>2$. The problem (\ref{eq:limit-U-0-1}) admits a unique bounded solution $U_\infty \in C^2(\overline I)$ with precisely $K-1$ zeros in $(0,\infty)$ and $U_\infty'(0)>0$. 
\end{proposition}

\begin{proof}
The existence of a bounded solution of \eqref{eq:limit-U-0-1} with precisely $K-1$ zeros in $(0,\infty)$ has been proved by Naito \cite[Theorem 1]{Naito}. To prove uniqueness, we first note that every solution $U$ of \eqref{eq:limit-U-0-1} is concave on intervals where $U>0$ and convex on intervals where $U<0$. From this we deduce that every bounded solution $U$ with finitely many zeros 
has a limit 
$$
\ell(U)= \lim_{t \to \infty}U(t) \not = 0.
$$
Next, we let $U_1$, $U_2$ be bounded solutions of \eqref{eq:limit-U-0-1} with precisely $K-1$ zeros in $(0,\infty)$. Moreover, we let 
$\kappa= \frac{\ell(U_1)}{\ell(U_2)}$, $c_\kappa:= \ln |\kappa|^{p-2}$ and consider 
$$
\tilde U_2: [c_\kappa, \infty) \to \R, \qquad \tilde U_2(t)= \kappa U_2(t-c_\kappa).
$$
Then $\tilde U_2$ solves the equation in \eqref{eq:limit-U-0-1} on $[c_\kappa, \infty)$ and satisfies $\tilde U_2(c_\kappa) = 0$. By construction we have 
$$
\lim_{t \to \infty}U_1(t)= \lim_{t \to \infty}\tilde U_2(t), 
$$
and thus the local uniqueness result at infinity given in \cite[Proposition 3.1]{Naito} implies that 
$$
U_1(t)= \tilde U_2(t) \qquad \text{for $t \ge \max \{0,c_\kappa\}$.}
$$
Since $U_1$ and $\tilde U_2$ have $K-1$ zeros in $(0,\infty)$, $(c_\kappa,\infty)$, respectively and $U_1(0)=\tilde U_2(c_\kappa)=0$, it follows that $c_\kappa=0$, hence $\kappa=1$ and therefore $U_1 \equiv U_2$. The uniqueness of $U_\infty$ thus follows. 
\end{proof}

In the following, it is more convenient to work with the parameter $\gamma= 
\frac{N-2}{N+\alpha} \in (0,\frac{N-2}{N})$ in place of $\alpha$. Hence, from now on, we will write $U_\gamma$ in place of $U_\alpha$. We also set $U_0:=(-1)^{K-1}U_\infty$, so that 
\begin{equation}
  \label{eq:lim-infty-U-0}
\lim_{t \to \infty}U_0(t)>0.  
\end{equation}
We wish to consider \eqref{eq:limit-U} and \eqref{eq:U-gamma} in suitable spaces of continuous functions. For $\delta \ge 0$, we let $C_\delta(\overline I)$ denote the space of all functions $v \in C(\overline I)$ such that
$$
\|v\|_{C_\delta} := \sup_{t \ge 0} e^{\delta t}|v(t)| < \infty,
$$
More generally, for an integer $k \ge 0$, we let 
$C_\delta^k(\overline I)$ denote the space of all functions 
$v \in C^k(\overline I)$ such that $v^{(j)} \in C_\delta(\overline I)$ for $j=1,\dots,k$. Then $C_\delta^k(\overline I)$ is a Banach space with norm 
$$
\|v\|_{C_\delta^k} := \sum_{j=0}^k \|v^{(j)}\|_{C_\delta} .
$$
We note the following. 
\begin{lemma}
  \label{compactness-C-delta-spaces}
Let $k > \ell \ge 0$ and $\delta_1 > \delta_2 \ge 0$. Then the 
embedding $C^{k}_{\delta_1}(I) \hookrightarrow C^{\ell}_{\delta_2}(I)$ is compact.
\end{lemma}

\begin{proof}
This is a straightforward consequence of the Arzel\`a-Ascoli theorem. 
\end{proof}

For the remainder of this section, we fix $\delta = \frac{2}{N}$ and consider the spaces 
$$
E  := \{ v \in C^2(\ov I)\::\: v(0)=0,\: v' \in C^1_{\delta}(I) \} \qquad \text{and}\qquad F:= C_{\delta}(I).
$$
As note above, $F$ is a Banach space with norm $\|\cdot\|_F = \|\cdot\|_{C_{\delta}}$. Moreover, for every $v \in E$ we have 
$$
|v(t)| \le \Bigl| \int_0^t v'(s)\,ds \Bigr| \le \|v'\|_{C^1_{\delta}} \int_0^t e^{-\frac{2s}{N}}\,ds \le \frac{N}{2} \|v'\|_{C^1_{\delta}} \qquad \text{for all $t \ge 0$}
$$ 
and therefore $\|v\|_{L^\infty(I)} \le \frac{N}{2} \|v'\|_{C^1_{\delta}}$. Hence we may endow $E$ with the norm 
$$
v \mapsto \|v\|_E  := \|v\|_{L^\infty(I)} + \|v'\|_{C^1_{\delta}}.
$$
Since $C^1_{\delta}$ is a Banach space, it easily follows that $E$ is a Banach space as well. We also note that
\begin{equation}
  \label{eq:limit-v}
\lim_{t \to \infty}v(t) = \int_0^\infty v'(s)\,ds \quad \text{exists for every $v \in E$.} 
\end{equation}

\begin{lemma}
\label{bounded-sol-E}
	Let $p>2$, $\gamma \in [0,\frac{N-2}{N}]$, and let $U \in C^2(\ov I)$ be a bounded nontrivial solution of \eqref{eq:U-gamma}. Then $U \in E$, and $\lim \limits_{t \to \infty}U(t) \not = 0$.
\end{lemma}
\begin{proof}
	Since $U$ is bounded, we have
	$$
	|(e^{-\gamma t}U')'| \leq e^{-t} |U|^{p-1} \leq C e^{-t} \qquad \text{for $t \ge 0$}
	$$
	with a constant $C>0$. Furthermore, there exists a sequence $t_n \to \infty$ with $U'(t_n) \to 0$ as $n \to \infty$. Consequently, 	$$
	e^{-\gamma t} |U'(t)| = \lim_{n \to \infty} \left| \int_t^{t_n} (e^{-\gamma s}U'(s))' \, ds \right| \leq \lim_{n \to \infty} C \int_t^{t_n} e^{- s} \, ds = C e^{-t} 
	$$
	and therefore $|U'(t)|  \leq C e^{(\gamma -1)t} \leq C e^{-\frac{2}{N} t}$ for $t \ge 0$. Since we can write \eqref{eq:U-gamma} as
	\begin{equation} \label{eq:U-gamma-alternate}
	-U'' + \gamma U' = e^{(\gamma-1)t} |U|^{p-2} U ,
	\end{equation}	
	it follows that
	$|U''(t)| \leq |\gamma| |U'(t)| + e^{(\gamma-1)t} |U(t)|^{p-1} \leq C' e^{-\frac{2}{N} t}$ for $t \ge 0$ with a constant $C'>0$, hence $U \in E$.

It remains to show that $\lim \limits_{t \to \infty}U(t) \not = 0$. For this we consider the nonincreasing function $m(t):= \sup \limits_{s \ge t}|U(s)|$. Using (\ref{eq:U-gamma}) and the fact that $U \in E$, we find that 
$$
e^{-\gamma t}|U'(t)| = \Bigl|\int_t^\infty e^{-s}|U(s)|^{p-2}U(s)\,ds\Bigr|
\le e^{-t} m^{p-1}(t) \qquad \text{for $t \ge 0$.}
$$
and therefore
$$
|U(t)| = \Bigl| \int_t^\infty U'(s)\,ds \Bigr| \le \int_{t}^\infty 
e^{(\gamma-1)s} m^{p-1}(s)\,ds \le \frac{m^{p-1}(t)}{1-\gamma}e^{(\gamma-1)t}\qquad \text{for $t \ge 0$.}  
$$
Consequently, 
$$
m(t) =  \sup_{s \ge t}|U(s)| \le \sup_{s \ge t}\Bigl(\frac{m^{p-1}(s)}{1-\gamma}e^{(\gamma-1)s}\Bigr) 
= \frac{m^{p-1}(t)}{1-\gamma}e^{(\gamma-1)t} 
$$
and hence $m(t)=0$ or $m^{p-2}(t) \ge (1-\gamma)e^{(1-\gamma)t} \ge 1-\gamma$ for $t \ge 0$. Since $m(0)\not = 0$ as $U \not \equiv 0$, 
we conclude by continuity of $m$ that $m^{p-2}(t)\ge 1-\gamma$ for all $t \ge 0$. Together with \eqref{eq:limit-v}, this shows that $\lim \limits_{t \to \infty}U(t) \not = 0$.
\end{proof}

We intend to use the implicit function theorem to show that $U_\gamma \to U_0$ in $E$ as $\gamma \to 0$. This requires uniqueness and nondegeneracy properties as given in the following two lemmas.

\begin{lemma} \label{uniqueness via zeros}
	Let $p>2$, $\gamma \in (0,\frac{N-2}{N+\alpha_p})$ and let $\tilde U \in E$ be a solution of \eqref{eq:U-gamma} with precisely $K-1$ zeros in $(0,\infty)$
	and $\lim \limits_{t \to \infty}\tilde U(t)>0$.
	Then $\tilde U = U_\gamma$. 
\end{lemma}
\begin{proof}
Let $\alpha>0$ be the unique value such that $\gamma = \gamma(\alpha)= \frac{N-2}{N+\alpha}$, and consider the function 
$$
u:[0,1] \to \R, \qquad 
u(r)= \left \{
  \begin{aligned}
&(N+\alpha)^\frac{2}{p-2} \tilde U(-(N+\alpha)\ln r),&&\qquad r>0,\\
&(N+\alpha)^\frac{2}{p-2} \lim \limits_{t \to \infty}\tilde U(t),&&\qquad r=0. 
  \end{aligned}
\right.
$$
Since $\tilde U \in E$, the latter limit exists. 
We then have $u \in C^2((0,1]) \cap C([0,1])$, and $u$ solves equation~\eqref{eq:u-alpha-radial-variable-1} on $(0,1)$. Moreover, we have  
$u'(r) = - (N+\alpha)^\frac{p}{p-2}\frac{\tilde U'(-(N+\alpha)\ln r)}{r}$ for $r \in (0,1]$ and therefore
$$
\lim_{r \to 0}\frac{u'(r)}{r} = - (N+\alpha)^\frac{2}{p-2}\lim_{t \to \infty} 
e^{\frac{2t}{N+\alpha}}  \tilde U'(t).
$$
Since $\frac{2}{N+\alpha}< \frac{2}{N}$ and $\tilde U \in E$, we deduce that $\lim \limits_{r \to 0}\frac{u'(r)}{r}=0$. From equation~\eqref{eq:u-alpha-radial-variable-1} it then also follows that $\lim \limits_{r \to 0}u''(r)$ exists, and that $u$ also satisfies the boundary conditions in~\eqref{eq:u-alpha-radial-variable-1}. Moreover, we have $u(0)>0$ since $\lim \limits_{t \to \infty}\tilde U(t)>0$ by assumption. The uniqueness result in Theorem \ref{sec:bifurc-nonr-solut-1} then yields that $u$ is equal to $u_\alpha$. Transforming back, we conclude that $\tilde U= U_\gamma$.
\end{proof}

\begin{lemma}
\label{sec:preliminaries-limit-problem-2}
Let $p>2$ and $\gamma \in [0,\frac{N-2}{N+\alpha_p})$. Then the solution $U_\gamma$ of problem (\ref{eq:U-gamma}) is nondegenerate in the sense that the equation 
$$
-(e^{-\gamma t} v')' -(p-1)e^{-t}|U_\gamma|^{p-2}v = 0 \quad \text{in $[0,\infty)$,}\qquad v(0)=0.
$$
has no bounded nontrivial solution.
\end{lemma}

\begin{proof}
We consider the auxiliary function $w:= U_\gamma' + \frac{\gamma-1}{p-2}U_\gamma$, which, by direct computation, solves the linearized equation 
\begin{equation}
  \label{eq:nondeg-1}
-(e^{-\gamma t} w')' -(p-1)e^{-t}|U_\gamma|^{p-2}w = 0 \qquad \text{in $[0,\infty)$.}
\end{equation}
Moreover, we have $\lim \limits_{t \to \infty}w'(t)=0$ since $U_\gamma \in E$ by Lemma~\ref{bounded-sol-E}. Suppose by contradiction there exists a bounded function $v \in C^2([0,\infty))$, $v \not \equiv 0$ satisfying 
\begin{equation}
  \label{eq:nondeg-2}  
-(e^{-\gamma t}v')' -(p-1)e^{-t}|U_\infty|^{p-2}v = 0 \quad \text{in $[0,\infty)$,}\qquad v(0)=0.
\end{equation}
Sturm comparison with $w$ yields that $v$ can only have finitely many zeros in $I$. Let $t_0>0$ denote the largest zero of $w$ in $[0,\infty)$. Since $v$ is bounded, there exists a sequence $(t_n)_n \subset [t_0,\infty)$ such that 
$t_n \to \infty$ and $v'(t_n) \to 0$ as $n \to \infty$. From (\ref{eq:nondeg-1}) and (\ref{eq:nondeg-2}), we deduce that 
$$
-\int_{t_0}^\infty (e^{-\gamma t}v')' w = \int_{t_0}^\infty e^{-t}|U_\infty|^{p-2} vw = - \int_{t_0}^\infty (e^{-\gamma t} w')' v .
$$
Since $\lim \limits_{n \to \infty} e^{-\gamma t_n} v'(t_n)=\lim \limits_{n \to \infty} e^{-\gamma t_n} w'(t_n)=0$, integration by parts yields
\begin{align*}
-e^{-\gamma t_0} v'(t_0)w(t_0) &= 
\lim_{n \to \infty} e^{-\gamma t_n} v'(t_n)w(t_n) - e^{-\gamma t_0} v'(t_0)w(t_0)\\
&= \lim_{n \to \infty} e^{-\gamma t_n} w'(t_n)v(t_n) - e^{-\gamma t_0} w'(t_0)v(t_0) = 0 ,
\end{align*}
which implies $v'(t_0)=0$ or $w(t_0)=0$. In the first case we then have $v \equiv 0$ and the proof is finished. In the other case it also follows that there exists $c \neq 0$ such that $cw'(t_0) = v'(t_0)$, which implies $v \equiv cw$. This contradicts $v(0)=0 \neq U_\infty'(0) = w(0)$.
\end{proof}

We may now state a continuation result for the map $\gamma \mapsto U_\gamma$ which in particular implies Proposition~\ref{limit-shape}.

\begin{proposition} \label{implicit function for U-gamma}
	Let $p>2$. There exists $\eps_0 >0$ such that the map $(0,\frac{N-2}{N+\alpha_p}) \to E$, $\gamma \mapsto U_\gamma$ extends to a $C^1$-map 
$g: (-\eps_0,\frac{N-2}{N+\alpha_p}) \to E$ with $g(0)=U_0$.
\end{proposition}

\begin{proof}
We consider the map
$$
G: \left(-\infty,\frac{N-2}{N+\alpha_p}\right) \times E  \to F,\qquad G(\gamma,U)= -U'' + \gamma U' - e^{(\gamma-1)t} |U|^{p-2} U.
$$
Since $e^{(\gamma-1)t} \le e^{-\frac{2}{N}t}$ for $\gamma<\frac{N-2}{N+\alpha_p}$, $G$ is well-defined and of class $C^1$. Moreover, by definition of $U_\gamma$ we have 
\begin{equation}
  \label{eq:G-first-equality}
G(\gamma,U_\gamma) =0  \qquad \text{for $\gamma \in \left[0,\frac{N-2}{N+\alpha_p}\right)$.}
\end{equation}
We first show that the linear map 
\begin{equation}
\label{L-gamma-isomorphism}  
L_\gamma :=d_U G(\gamma,U_\gamma): E  \to F,\qquad L \phi= -\phi'' + \gamma \phi' - (p-1)e^{(\gamma-1)t}|U_\gamma|^{p-2} \phi 
\end{equation}
is an isomorphism for $\gamma \in [0,\frac{N-2}{N+\alpha_p})$. For this, we first note that 
\begin{equation}
 \label{sec-3-isomorphism}
\text{the map $E  \to F$, $\phi \mapsto -\phi'' +\gamma \phi'$ is an isomorphism.} 
\end{equation}
Indeed, if $\phi \in E$ satisfies $-\phi''+ \gamma \phi' =0$, then $-\phi'+ \gamma \phi$ is constant and $\phi(0)=0$, hence   
$\phi(t)= c(e^{\gamma t}-1)$ for $t \in I$ with a constant $c \in \R$. Since $\phi \in E \subset L^\infty(I)$, we conclude that $\phi \equiv 0$.
 
Moreover, if $f \in F$ is given and $\phi: I \to \R$ is defined by 
	$$
	\phi(t):=  \int_0^t \int_s^\infty e^{\gamma(s-\sigma)}f(\sigma) \, d\sigma ds ,
	$$
	we have $-\phi''+ \gamma \phi'=f$ and $\phi(0)=0$. Furthermore, 
$$
|\phi'(t)|=  \Bigl|\int_t^\infty e^{\gamma(t-\sigma)}f(\sigma) \, d\sigma\Bigr|  \leq \int_t^\infty |f(\sigma)| \, d\sigma \leq \|f\|_F 
\int_t^\infty e^{-\frac{2}{N} s} \, ds \leq  \frac{N}{2} \|f\|_F\: e^{-\frac{2}{N} t}
$$ 
for $t \ge 0$ and therefore $\phi \in E$. We thus infer (\ref{sec-3-isomorphism}).
 
Next, we note that the linear map $E \to F$, $\phi \mapsto e^{(\gamma-1)(\cdot)} |U_0|^{p-2} \phi$ is compact, since the embedding $E \hookrightarrow C_0(I)$ is compact by Lemma~\ref{compactness-C-delta-spaces} and the map 
$C_0(I) \to F$, $\phi \mapsto e^{(\gamma-1)(\cdot)} |U_0|^{p-2} \phi$ is continuous. By \eqref{sec-3-isomorphism}, we therefore deduce that $L$ is Fredholm  of index zero. Since the equation $L_\gamma v=0$ only has the trivial solution $v=0$ in $E$ by Lemma~\ref{sec:preliminaries-limit-problem-2}, we conclude that $L_\gamma$ is an isomorphism, as claimed.
We now apply the implicit function theorem to the map $G$ in the point $(0,U_0)$. This yields $\eps_0>0$ and a differentiable map ${\tilde g}:(-\eps_0,\eps_0) \to E$ with ${\tilde g}(0)=U_0$ and $G(\gamma,{\tilde g}(\gamma)) = 0$ for $\gamma \in (-\eps_0,\eps_0)$.\\
Next we claim that 
\begin{equation}
\label{convergence in E-prop}
\text{$U_\gamma = {\tilde g}(\gamma)$ for $\gamma \in [0,\eps_0)$.}  
\end{equation}
Indeed, let $v_\gamma: = {\tilde g}(\gamma) \in E$ for $\gamma \in (-\eps_0,\eps_0)$. By the continuity of ${\tilde g}: (-\eps_0,\eps_0) \to E$ and (\ref{eq:limit-v}), the function 
$$
(-\eps_0,\eps_0) \to \R, \qquad \gamma \mapsto  m_\gamma:= \lim_{t \to \infty}v_\gamma(t)
$$
is also continuous, and it is nonzero for $\gamma \in [0,\eps_0)$ by Lemma~\ref{bounded-sol-E}. Moreover, by construction we have $v_0 = U_0$ and therefore $m_0 > 0$. It then follows that 
\begin{equation}
  \label{eq:m-gamma-pos}
m_\gamma >0 \qquad \text{for all 
$[0,\eps_0)$.} 
\end{equation}
By Lemma~\ref{uniqueness via zeros}, we thus only need to prove that $v_\gamma$ has $K-1$ zeros in $(0,\infty)$ for $\gamma \in [0,\eps_0)$. This is true for $\gamma= 0$ since $v_0 = U_0$. Moreover, the number of zeros of $v_\gamma$ remains constant for $\gamma \in [0,\eps_0)$. Indeed, as a solution of (\ref{eq:U-gamma}), $v_\gamma$ cannot have double zeros, and the largest zero $t_\gamma$ of $v_\gamma$ in $[0,\infty)$ remains locally bounded for $\gamma \in [0,\eps_0)$ since
$$
m_\gamma = \int_{t_\gamma}^\infty v_\gamma'(s)\,ds \le \|v_\gamma\|_{E}
 \int_{t_\gamma}^\infty e^{-\frac{2}{N} s}\,ds \le \frac{N}{2}\|v_\gamma\|_E \,e^{-\frac{2}{N} t_\gamma}.
$$
and therefore $t_\gamma \le - \frac{N}{2} \ln \frac{2 m_\gamma}{N \|v_\gamma\|_E}
$. This finishes the proof of (\ref{convergence in E-prop}).

By a continuation argument based on (\ref{L-gamma-isomorphism}), an application of the implicit function theorem in points $(\gamma,U_\gamma)$ for $\gamma>0$ and the same continuity considerations as above , we then see that the map 
$$
g: (-\eps_0,\frac{N-2}{N+\alpha_p}) \to E, \qquad g(\gamma) = \left\{
  \begin{aligned}
  &\tilde g(\gamma),&&\qquad \gamma \in (-\eps_0,0),\\
  &U_\gamma,&&\qquad \gamma \in \left[0,\frac{N-2}{N+\alpha_p} \right)  
  \end{aligned}
\right.
$$
is of class $C^1$. The proof is thus finished. 
\end{proof}

Since $U_0 = (-1)^{K-1} U_\infty$, we have now completed the proof of Proposition~\ref{limit-shape}.

\begin{remark}
\label{gamma-negative-definition}
	Using the function $g$ and $\eps_0>0$ from Proposition~\ref{implicit function for U-gamma}, it is convenient to define
	$$
	U_\gamma := g(\gamma) \quad \text{for $\gamma \in (-\eps_0,0).$}
	$$ 
With this definition, it follows from Proposition~\ref{implicit function for U-gamma} 
that the map $(-\eps_0,\frac{N-2}{N+\alpha_p}) \to E$, $\gamma \mapsto U_\gamma$ is of class $C^1$.

Moreover, implicit differentiation of (\ref{eq:U-gamma}) at $\gamma=0$ shows that $V= \partial_\gamma \big|_{\gamma=0} U_\gamma$ is given as the unique bounded solution of the problem 
\begin{equation}
\label{implicit-differentiation-formula-1}
-V'' - (p-1)e^{-t}|U_0|^{p-2}V = U_0' - t e^{-t}|U_0|^{p-2}U_0 \quad \text{in $[0,\infty)$,}\qquad V(0)=0.  
\end{equation}
 
\end{remark}

\section{Spectral asymptotics}
\label{sec:spectr-asympt-refeq}

This section is devoted to the proofs of Theorem~\ref{spectral-curves}
and Corollary~\ref{corollary-on-eigenvalue-curves}. We fix $p>2$, and we start by recalling some results from \cite{Amadori-Gladiali2} on the eigenvalue problem (\ref{eq:weighted-L-alpha-problem}) and its relationship to the Morse index of $u_\alpha$. Recall that we consider (\ref{eq:weighted-L-alpha-problem}) in weak sense. More precisely, we say that $\phi \in H^1_0(\B)$ is an eigenfunction of (\ref{eq:weighted-L-alpha-problem}) corresponding to the eigenvalue $\lambda \in \R$ if 
\begin{equation}
  \label{eq:weak-eigenvalue-singular}
q_\alpha(\phi,\psi) = \lambda \int_{\B} \frac{\phi(x) \psi(x)}{|x|^2}dx \qquad \text{for all $\psi \in H^1_0(\B)$,}
\end{equation}
where 
\begin{equation}
  \label{eq:def-q-alpha}
q_\alpha: H^1_0(\B)\times H^1_0(\B) \to \R, \qquad q_\alpha(v,w) := \int_{\B} \Bigl(\nabla v \cdot \nabla w  - (p-1)|x|^\alpha |u_\alpha|^{p-2} vw\Bigr) \, dx 
\end{equation}
is the quadratic form associated with the operator $L^\alpha$. Note that the RHS of (\ref{eq:weak-eigenvalue-singular}) is well-defined for $\phi, \psi \in H^1_0(\B)$ by Hardy's inequality.

\begin{lemma}(see \cite[Prop. 4.1 and 5.1]{Amadori-Gladiali2})\\
  \label{sec:spectral-asymptotics-morse-index}
Let $\alpha>\alpha_p$. Then we have: 
\begin{enumerate}
\item[(i)] The Morse index of $u_\alpha$ is given as the number of negative eigenvalues of \eqref{eq:weighted-L-alpha-problem}, counted with multiplicity. Moreover, every eigenfunction $v \in H^1_0(\B)$ of (\ref{eq:weighted-L-alpha-problem}) corresponding to a nonpositive eigenvalue is contained in $L^\infty(\B) \cap C^2(\B \setminus \{0\})$. 
\item[(ii)] Let $\phi \in H^1_0(\B)$ be an eigenfunction of (\ref{eq:weighted-L-alpha-problem}) corresponding to the eigenvalue $\lambda \in \R$. 
Then there exists a number $\ell_0 \in \N \cup \{0\}$, spherical harmonics $Y_\ell$ of degree $\ell$ and functions $\phi_\ell \in H^1_{0,rad}(\B)$, $\ell=1,\dots,\ell_0$ with the property that 
$$
\phi(x)= \sum_{\ell =0}^{\ell_0} \phi_\ell(x) Y_\ell\left(\frac{x}{|x|}\right) \qquad \text{for $x \in \B$.}
$$  
Moreover, for every $\ell \in \{1,\dots,\ell_0\}$, we either have $\phi_\ell \equiv 0$, or $\phi_\ell$ is an eigenfunction of (\ref{eq:weighted-eigenvalue-problem-reduced}) corresponding to the eigenvalue $\mu = \lambda-\lambda_\ell$. 
\end{enumerate}
\end{lemma}

Regarding the reduced weighted eigenvalue problem~(\ref{eq:weighted-eigenvalue-problem-reduced}), we also recall the following. 
 
\begin{lemma} (see \cite[p.19 and Prop. 3.7]{Amadori-Gladiali2})\\
\label{eigenvalue lemma}
Let $\alpha>\alpha_p$. Then $0$ is not an eigenvalue of \eqref{eq:weighted-eigenvalue-problem-reduced}, and the negative eigenvalues of \eqref{eq:weighted-eigenvalue-problem-reduced} are simple and given by 
\begin{equation}
  \label{eq:var-char-mu-j}
\mu_j(\alpha):= \inf_{\substack{W \subset H_{0,rad}^1(\B)\\ \dim W=j}} \max_{v \in W\setminus \{0\}}  \frac{\int_{\B} |\nabla v|^2 - (p-1)|x|^\alpha |u_\alpha|^{p-2} |v|^2 \, dx  }{\int_\B |x|^{-2}|v|^2 \, dx }, \qquad j= 1,\dots, K.
\end{equation}
\end{lemma}

Here we point out that Theorem~\ref{sec:bifurc-nonr-solut-1}(i) already implies that zero is not an eigenvalue of \eqref{eq:weighted-eigenvalue-problem-reduced}. We also note that Proposition~\ref{spectral-curves-0} now merely follows by combining Lemma \ref{sec:spectral-asymptotics-morse-index} and Lemma~\ref{eigenvalue lemma}. 
 
We now turn to the proof of Theorem \ref{spectral-curves}. For this we transform the radial eigenvalue problem (\ref{eq:weighted-eigenvalue-problem-reduced}). Note that, if we write an eigenfunction $\psi \in H^1_{0,rad}(\B)$ as a function of the radial variable $r = |x|$, it solves 
$$
-\psi'' -\frac{N-1}{r}\psi' - (p-1)r^{\alpha}|u_\alpha(r)|^{p-2}\psi(r) = \frac{\mu}{r^2} \psi \qquad \text{in $(0,1)$,}\qquad \qquad  \psi(1)=0.
$$
We transform this problem by considering again $I:=(0,\infty)$ and setting 
\begin{equation}
  \label{eq:transformed-variables}
\nu = \frac{1}{(N+\alpha)^2} \mu_j(\alpha), \qquad \qquad \Psi(t)=(N+\alpha) \psi(e^{- \frac{t}{N+\alpha}})\quad \text{for $t \in \overline I$.}
\end{equation}
This gives rise to the eigenvalue problem
\begin{equation}
  \label{eq:trans-weighted-eigenvalue-problem}
\left\{
  \begin{aligned}
&-(e^{-\gamma t}\Psi')' - (p-1)e^{-t}|U_\gamma(t)|^{p-2}\Psi= 
\nu e^{-\gamma t} \Psi \quad \text{in $I$,}\\
&\qquad \Psi(0)=0, \quad \Psi \in L^\infty(I)
\end{aligned}
\right.
\end{equation}
with $\gamma= \gamma(\alpha)=\frac{N-2}{N+\alpha} \in (0,\frac{N-2}{N+\alpha_p})$ as before. Here, we have added the condition $\Psi \in L^\infty(I)$ since we focus on
eigenfunctions corresponding to negative eigenvalues, and in this 
case eigenfunctions $\psi \in H^1_{0,rad}(\B)$ of (\ref{eq:weighted-eigenvalue-problem-reduced}) are bounded by Lemma~\ref{eigenvalue lemma}. In the following, we also consider the case $\gamma=0$ in (\ref{eq:trans-weighted-eigenvalue-problem}), which corresponds to the linearization of (\ref{eq:limit-U-0-1}) at $U_0$:
\begin{equation}
  \label{eq:weighted-eigenvalue-translimit}
\left\{
  \begin{aligned}
&-\Psi'' - (p-1)e^{-t}|U_0(t)|^{p-2}\Psi= 
\nu \Psi \quad \text{in $I$,}\\
&\qquad \Psi(0) = 0, \quad \Psi \in L^\infty(I).
\end{aligned}
\right.
\end{equation}
We note that for $\gamma \in [0,\frac{N-2}{N+\alpha})$ and every solution $\Psi$ of (\ref{eq:trans-weighted-eigenvalue-problem}) there exists a sequence $t_n \to \infty$ with $\Psi'(t_n) \to 0$, which implies that 
\begin{equation}
  \label{eq:int-formular-eigenvalue}
e^{-\gamma t}\Psi'(t) = \int_{t}^\infty -(e^{-\gamma s}\Psi')'(s)\,ds = 
\int_{t}^\infty \bigl(\nu e^{-\gamma s}  + (p-1)e^{-s}|U_\gamma(s)|^{p-2}\bigr)\Psi(s)\,ds
\end{equation}
for $t \ge 0$. We also note that problem (\ref{eq:trans-weighted-eigenvalue-problem}) can be rewritten as 
\begin{equation}
  \label{eq:trans-weighted-eigenvalue-problem-rewritten}
\left\{
  \begin{aligned}
&-\Psi'' + \gamma \Psi'-  (p-1)e^{(\gamma-1)t}|U_\gamma(t)|^{p-2}\Psi= 
\nu \Psi \quad \text{in $I$,}\\
&\qquad \Psi(0)=0, \quad \Psi \in L^\infty(I).
\end{aligned}
\right.
\end{equation}
We need the following estimate in terms of the space $C_\delta^2(I)$ defined in Section~\ref{sec:limit-alpha-to}. 

\begin{lemma}
\label{bounded-sol-exp-decay} 
Let $\nu_\sdiamond<0$, $\gamma_\sdiamond \in (0,\frac{N-2}{N+\alpha_p})$, and let $\delta = \frac{1}{2}\bigl(\sqrt{1-2\nu_\sdiamond}-1\bigr)>0$. Then there exists a constant $C=C(\nu_\sdiamond,\gamma_\sdiamond)>0$ such that for every solution $\Psi \in L^\infty(I)$ of the equation
\begin{equation}
  \label{eq:bounded-sol-decay-eqn}
-\Psi'' + \gamma \Psi'-  (p-1)e^{(\gamma-1)t}|U_\gamma(t)|^{p-2}\Psi= 
\nu \Psi 
\end{equation}
with $\nu \le \nu_\sdiamond$ and $\gamma \in [0,\gamma_\sdiamond]$ we have $\Psi \in C_\delta^2(I)$ with 
$\|\Psi\|_{C^2_\delta} \le C \|\Psi\|_{L^\infty(I)}$.
\end{lemma}

\begin{proof}
Since $\|U_{\gamma}\|_{L^\infty(I)}$ remains uniformly bounded for $\gamma \in [0,\gamma_\sdiamond]$ by Proposition~\ref{implicit function for U-gamma}, there exists $t_0= t_0(\nu_\sdiamond,\gamma_\sdiamond)>0$ such that 
$$
(p-1) e^{(\gamma-1)t} |U_{\gamma}(t)|^{p-2} \leq -\frac{\nu_\sdiamond}{2} \quad \text{for $t \geq t_0$, $\gamma \in \left[0,\gamma_\sdiamond\right]$.}
$$
Let $\Psi$ be a bounded solution of (\ref{eq:bounded-sol-decay-eqn}) on $I$. Then $\Psi$ solves the differential inequality 
\begin{equation}
  \label{eq:psi-diff-ineq}
\Psi'' - \gamma \Psi' + \frac{\nu_\sdiamond}{2} \Psi \ge 0 \qquad \text{in the open set $U_\Psi:= \{t \in (t_0,\infty)\,:\, \Psi(t)>0\}.$} 
\end{equation}
For fixed $\eps>0$, we consider the function
$$
t \mapsto \phi_\eps(t):= C_\Psi e^{-\delta t} + \eps e^{\delta t} \qquad \text{with $C_\Psi:= e^{\delta t_0} \|\Psi\|_{L^\infty(I)}$.}
$$
By (\ref{eq:psi-diff-ineq}) and the definition of $\delta$, the function $v_\eps:= \phi_\eps -\Psi$ satisfies 
\begin{align*}
v_\eps'' - \gamma v_\eps'  + \frac{\nu_\sdiamond}{2} v_\eps  \le  (\delta^2 +  \frac{\nu_\sdiamond}{2})\phi_\eps
+ \gamma \delta C_\Psi e^{-\delta t} - \gamma \delta \eps e^{\delta t} &
\le (\delta^2 +|\gamma| \delta + \frac{\nu_\sdiamond}{2})\phi_\eps\\
&\le (\delta^2 + \delta + \frac{\nu_\sdiamond}{2})\phi_\eps = 0 \qquad \text{in $U_\Psi$.}
\end{align*}
This implies that $v_\eps$ cannot attain a negative minimum in the set $(t_0,\infty)$. Moreover, by definition of $v_\eps$ we have 
$$
v_\eps(t_0) \ge 0 \qquad \text{and}\qquad \lim_{t \to \infty}v_\eps(t)= \infty.
$$
Consequently, we have $v_\eps \ge 0$ and therefore 
$\Psi \le \phi_\eps$ on $[t_0,\infty)$. Replacing $\Psi$ by $-\Psi$ in the argument above, we find that $|\Psi| \le \phi_\eps$ on $[t_0,\infty)$. By considering the limit $\eps \to 0$, we deduce that 
$$
|\Psi(t)| \le C_\Psi e^{-\delta t} = C \|\Psi\|_{L^\infty(I)} e^{-\delta t} \qquad \text{for $t \ge t_0$ with $C:= e^{\delta t_0}$.} 
$$
Since the same inequality obviously holds for $t \in [0,t_0)$, we conclude that 
$$
|\Psi(t)| \le C\|\Psi\|_{L^\infty(I)} e^{-\delta t} \qquad \text{for $t \ge 0$.} 
$$
Finally, using (\ref{eq:int-formular-eigenvalue}) and (\ref{eq:bounded-sol-decay-eqn}), we also get that 
$$
|\Psi'(t)| \le C\|\Psi\|_{L^\infty(I)} e^{-\delta t} \quad \text{and} \quad |\Psi''(t)| \le C\|\Psi\|_{L^\infty(I)} e^{-\delta t} \qquad \text{for $t \ge 0$} 
$$
after making $C>0$ larger if necessary. The proof is thus finished. 
\end{proof}

\begin{proposition}
\label{sec:preliminaries-limit-problem-3}
For $\gamma \in [0,\frac{N-2}{N+\alpha_p})$, the eigenvalue problem~(\ref{eq:trans-weighted-eigenvalue-problem}) admits precisely $K$ negative eigenvalues $\nu_1(\gamma) < \nu_2(\gamma) < \dots< \nu_K(\gamma) < 0$ characterized variationally by 
\begin{equation}
\label{var-char-nu-j} 
\nu_j(\gamma)  = \inf_{\substack{W \subset H^1_0(I)\\ \dim W=j}} \max_{\Psi \in W\setminus \{0\}}  \frac{\int_0^\infty e^{-\gamma t}\Psi'^2 -  (p-1)e^{-t}|U_\gamma|^{p-2}\Psi^2 \, dt}{\int_0^\infty e^{-\gamma t} \Psi^2 \, dt} \qquad \text{for $j=1,\dots,K$.}
\end{equation}
\end{proposition}

\begin{proof}
Let $\gamma \in [0,\frac{N-2}{N+\alpha_p})$. We first show that 
\begin{equation}
  \label{eq:gamma-K-less-zero}
\nu_K(\gamma)<0.  
\end{equation}
For $\gamma>0$, this follows by Lemma~\ref{eigenvalue lemma}. Indeed, in (\ref{eq:var-char-mu-j}) we may, by density, replace $H^1_{0,rad}(\B)$ by the space of radial functions in $C^\infty_c(\B \setminus \{0\})$, and this space corresponds to the dense subspace $C^\infty_c(I) \subset H^1_0(I)$ after the transformation (\ref{eq:transformed-variables}). To show (\ref{eq:gamma-K-less-zero}) in the case $\gamma=0$, we use the auxiliary function $w:= U_0' -\frac{1}{p-2}U_0$, which, by direct computation, solves the linearized equation $-w'' -(p-1)e^{-t}|U_0|^{p-2}w = 0$ in $(0,\infty)$. It is clear that $w$ has a zero between any two zeros of $U_0$ on $[0,\infty)$.
Moreover, letting $t_*>0$ denote the largest zero of $U_0$, we find that 
the numbers 
$$
w(t_*)=U_0'(t_*)\qquad \text{and}\qquad  \lim_{t \to \infty}w(t)=-\frac{1}{p-2}\lim_{t \to \infty}U_0(t)
$$
have opposite sign, hence $w$ also has a zero in $(t_*,\infty)$. Since $U_0$ has $K-1$ zeros in $(0,\infty)$ and $U_0(0)=0$, we infer that $w$ has at least $K$ zeros in $(0,\infty)$. From this, it is standard to deduce that $\nu_K(0)<0$. We thus have proved (\ref{eq:gamma-K-less-zero}).

Next we note that eigenfunctions $\Psi$ of (\ref{eq:trans-weighted-eigenvalue-problem}) corresponding to an eigenvalue 
$\nu_j(\gamma)<0$ have precisely $j-1$ zeros in $I$. Indeed, this follows from standard Sturm-Liouville theory since any such eigenfunction decays exponentially as $t \to \infty$ together with their first and second derivatives by Lemma~\ref{bounded-sol-exp-decay}. It also follows that $\nu_j(\gamma)$ is simple in this case, i.e., the corresponding eigenspace is one-dimensional.

In the case $\gamma>0$, the claim now follows from Lemma~\ref{eigenvalue lemma}, which guarantees that $\nu_1(\gamma),\dots,\nu_K(\gamma)$ are precisely the negative eigenvalues of (\ref{eq:trans-weighted-eigenvalue-problem-rewritten}). It remains to show that (\ref{eq:weighted-eigenvalue-translimit}) has precisely $K$ negative eigenvalues given by (\ref{var-char-nu-j}) in the case $\gamma=0$. Since the essential spectrum of the linearized operator $L_0: H^2(I) \cap H^1_0(I) \to L^2(I)$, $L_0 \Psi = -\Psi'' - (p-1)e^{-t}|U_0(t)|^{p-2}\Psi$ is given by $[0,\infty)$, standard compactness arguments 
show that $\nu_j(0)$ is an eigenvalue of (\ref{eq:weighted-eigenvalue-translimit}) whenever $\nu_j(0)<0$. Suppose by contradiction that $\nu_{K+1}(0)<0$, and let $v$ be a corresponding eigenfunction. Then $v$ has $K$ zeros in $(0,\infty)$, and $\lim \limits_{t \to \infty}v(t)=\lim \limits_{t \to \infty}v'(t)= 0$ as $t \to \infty$ by Lemma~\ref{bounded-sol-exp-decay}. By Sturm comparison, it then follows that $w$ has at least $K+1$ zeros in $(0,\infty)$. On the other hand, since 
$$
\Bigl(e^{-t}|U_0|^{p-2}+\frac{1}{(p-2)^2}\Bigr)U_0 =-U_0''+\frac{1}{(p-2)^2}U_0 = 
-w' - \frac{1}{p-2}w,
$$
$U_0$ has a zero between any two zeros of $w$. This contradicts the fact that $U_0$ has precisely $K-1$ zeros in $(0,\infty)$. We thus conclude that (\ref{eq:weighted-eigenvalue-translimit}) admits precisely $K$ negative eigenvalues given by (\ref{var-char-nu-j}) in the case $\gamma=0$.  
\end{proof}

We may now deduce the continuous dependence of the negative eigenvalues of (\ref{eq:trans-weighted-eigenvalue-problem}).

\begin{lemma} \label{eq:gamma-sufficient} 
	For $j=1,\dots,K$, the function $\nu_j: [0,\frac{N-2}{N+\alpha_p}) \to (-\infty,0)$ is continuous. 
\end{lemma}
\begin{proof}
Let $\gamma_0 \in [0,\frac{N-2}{N+\alpha_p})$, and let $(\gamma_n)_n \subset [0,\frac{N-2}{N+\alpha_p})$ be a sequence with $\gamma_n \to \gamma_0$. 
Recall that $U_{\gamma_n} \to  U_{\gamma_0}$ uniformly on $[0,\infty)$ as $n \to \infty$ by Proposition~\ref{implicit function for U-gamma}. We fix $j \in \{1,\ldots,K\}$ and consider the space $W \subset H^1_0(I)$ spanned by the first $j$ eigenfunctions of (\ref{eq:trans-weighted-eigenvalue-problem}) in the case $\gamma= \gamma_0$. Moreover, we let $\cM:= \{\Psi \in W\::\: \int_0^\infty \Psi^2dt =1\}$. Since $\nu_j(\gamma_0)<0$, $\cM$ is a compact subset of $C^2_\delta(I)$ for some $\delta>0$ by Lemma~\ref{bounded-sol-exp-decay}. From this we deduce that  
\begin{align*}
&\int_0^\infty \Bigl( e^{-\gamma_n t}{\Psi'}^2 -  (p-1)e^{-t}|U_{\gamma_n}|^{p-2}\Psi^2 \,\Bigr) dt
 \;\to \;  \int_0^\infty \Bigl(e^{-\gamma_0 t} {\Psi'}^2 -  (p-1)e^{-t}|U_0|^{p-2}\Psi^2\Bigr) \, dt \quad \text{and}\\
&\int_0^\infty e^{-\gamma_n t} \Psi^2 \, dt \;\to\; \int_0^\infty e^{-\gamma_0 t}\Psi^2 \, dt
\qquad \text{as $n \to \infty$ uniformly in $\psi \in \cM$,}
\end{align*}
and this implies that 
\begin{align*}
\limsup_{n \to \infty} \nu_j(\gamma_n) &\leq \limsup_{n \to \infty} \max_{\Psi \in \cM}
\frac{\int_0^\infty \Bigl( e^{-\gamma_n t}{\Psi'}^2 -  (p-1)e^{-t}|U_{\gamma_n}|^{p-2}\Psi^2\Bigr) \, dt}{\int_0^\infty e^{-\gamma_n t} \Psi^2 \, dt}\\
&= \max_{\Psi \in \cM}
\frac{\int_0^\infty \Bigl( e^{-\gamma_0 t} {\Psi'}^2 -  (p-1)e^{-t}|U_0|^{p-2}\Psi^2\Bigr) \, dt}{\int_0^\infty e^{-\gamma_0 t}\Psi^2 \, dt} = \nu_j(\gamma_0) .
\end{align*}
To show that $\liminf \limits_{n \to \infty} \nu_j(\gamma_n) \ge \nu_j(\gamma_0)$, we argue by contradiction and assume that, after passing to a subsequence,  we have 
\begin{equation} \label{lower eigenvalue convergence}
\nu_j(\gamma_n) \to \sigma_j < \nu_j(\gamma_0) .
\end{equation}
Passing again to a subsequence, we may then also assume that 
\begin{equation} 
\label{lower eigenvalue convergence-1}
\nu_k(\gamma_n) \to \sigma_k \le \sigma_j <0 \qquad \text{for $k=1,\dots,j$.}
\end{equation}
Let, for $k=1,\dots,j$, the function $\Psi_{k,n}$ denote an eigenfunction of (\ref{eq:trans-weighted-eigenvalue-problem}) corresponding to the eigenvalue $\nu_k(\gamma_n)$ such that $\|\Psi_{k,n}\|_{L^\infty(I)}=1$. Since eigenfunctions corresponding to different eigenvalues are orthogonal with respect to the weighted scalar product $(v,w) \mapsto \int_{0}^\infty e^{-\gamma_n t} vw\,dt$, we may assume that \begin{equation}
  \label{eq:proof-orthogonal}
\int_{0}^\infty e^{-\gamma_n t}\Psi_{k,n} \Psi_{\ell,n}\,dt = 
0 \qquad \text{for $k,\ell \in \{1,\dots,j\}$, $k \not = \ell.$}
\end{equation}
By Lemma~\ref{bounded-sol-exp-decay} and (\ref{lower eigenvalue convergence-1}), 
there exists $\delta>0$ such that $\|\Psi_{k,n}\|_{C^2_\delta} \le C$ for all $n \in \N$, $k \in \{1,\dots,j\}$. By Lemma~\ref{compactness-C-delta-spaces}, we may therefore pass to a subsequence again such that 
$$
\Psi_{k,n} \to \Psi_{k}  \qquad \text{uniformly in $I$,} 
$$
where $\Psi_k \in C^2_\delta(I)$ is a solution of 
\begin{equation}
  \label{eq:limit-eq-proof}
-(e^{\gamma_0 t}\Psi')' - (p-1)e^{-t}|U_0(t)|^{p-2}\Psi= \sigma_k e^{-\gamma_0 t}\Psi\qquad \text{in $I$},\qquad 
\Psi_k(0)=0
\end{equation}
for $k = 1,\dots,j$. Moreover, since the sequences $(\Psi_{k,n})_n$, $k = 1,\dots,j$ are uniformly bounded in $C^2_\delta(I)$, we may  
pass to the limit in (\ref{eq:proof-orthogonal}) to get that 
\begin{equation}
  \label{eq:proof-orthogonal-1}
\int_{0}^\infty e^{\gamma_0 t} \Psi_{k} \Psi_{\ell}\,dt = 
0 \qquad \text{for $k,\ell \in \{1,\dots,j\}$, $k \not = \ell.$}
\end{equation}
Consequently, for $\gamma=\gamma_0$, the problem (\ref{eq:trans-weighted-eigenvalue-problem}) has $j$ eigenvalues $\sigma_1,\dots,\sigma_j$ (counted with multiplicity) in $(-\infty,\nu_j(\gamma_0))$. This contradictions Proposition~\ref{sec:preliminaries-limit-problem-3}. The proof is finished.
\end{proof}

Next, we wish to derive some information on the derivative $\partial_\gamma \nu_j(\gamma)$ of the negative eigenvalues of (\ref{eq:trans-weighted-eigenvalue-problem}) as $\gamma \to 0^+$. We intend to derive this information via the implicit function theorem applied to the map $G: \left(-\eps_0, \frac{N-2}{N+\alpha_p}\right) \times \tilde E  \times \R  \to \tilde F \times \R$ defined by 
\begin{equation}
  \label{def-G-implicit}
G(\gamma, \Psi, \nu)=  \begin{pmatrix}
-\Psi''+\gamma \Psi' - (p-1)e^{(\gamma-1)t}|U_{\gamma}|^{p-2}\Psi - \nu \Psi \\
\int_0^\infty \Psi^2 \, dt -1,  
\end{pmatrix}
\end{equation}
Here, $\eps_0$ is given in Proposition~\ref{implicit function for U-gamma}, so that $(-\eps_0,\frac{N-2}{N+\alpha_p}) \to C^1_0(I)$, $\gamma \mapsto U_\gamma$ is a well defined $C^1$-map by Remark~\ref{gamma-negative-definition}. Moreover, $\tilde E$ and $\tilde F$ are suitable spaces of functions on $I$ chosen in a way that eigenfunctions and eigenvalues of \eqref{eq:trans-weighted-eigenvalue-problem-rewritten} and \eqref{eq:weighted-eigenvalue-translimit} correspond to zeros of this map. 
However, in the case $p \in (2,3]$, the function $|\cdot|^{p-2}$ is not differentiable at zero and therefore it is not a priori clear how $\tilde E$ and $\tilde F$ need to be chosen to guarantee that $G$ is of class $C^1$. 
In particular, spaces of continuous functions will not work in this case, so we need to introduce different function spaces. 

For $\delta>0$ and $1 \le r < \infty$, we let $L^r_\delta(I)$ denote the space of all functions 
$f \in L^r_{loc}(I)$ such that 
$$
\|f\|_{r,\delta} := \sup_{t \ge 0}e^{\delta t}[f]_{t,r} < \infty,\qquad \text{where}\quad
[f]_{t,r} : = \Bigl(\int_{t}^{t+1}|f(s)|^r\,ds\Bigr)^{\frac{1}{r}}= \|f\|_{L^r(t,t+1)}.
$$
The completeness of $L^r$-spaces readily implies that the spaces $L^r_\delta(I)$ are also Banach spaces. We will need the following observation: 
\begin{lemma}
\label{sec:case-2p3-1-lemma}  
Let $\delta>0$ and $f \in L^1_\delta(I)$. Then we have 
\begin{equation}
\int_{t}^\infty e^{\mu s} |f(s)|\,ds \le C_{\mu,\delta}\|f\|_{1,\delta}\: e^{(\mu-\delta) t} \qquad \text{for $\mu < \delta$, $t \ge 0$ with $C_{\mu,\delta}:= \frac{\max\{1,e^{\mu}\}}{1-e^{\mu-\delta}}$}  
\label{eq:L-1-delta-est}
\end{equation}
 and 
\begin{equation}
\int_{0}^t e^{\mu s} |f(s)|\,ds \le  D_{\delta,\mu}  \|f\|_{1,\delta}\: e^{(\mu-\delta)t}\qquad \text{for $\mu> \delta$, $t \ge 0$  with $D_{\delta,\mu}:=  \frac{e^{2\mu-\delta}}{e^{\mu-\delta}-1}$.}
\label{eq:L-1-delta-est-1}
\end{equation}
\end{lemma}

\begin{proof}
Let $f \in L^1_\delta(I)$ and $t \ge 0$. If $\mu<\delta$, we have 
\begin{align}
&\int_{t}^\infty e^{\mu s} |f(s)|\,ds = \sum_{\ell= 0}^\infty \int_{t+\ell}^{t+\ell+ 1}e^{\mu s} |f(s)|\,ds \le \max\{1,e^\mu\} \sum_{\ell= 0}^\infty e^{\mu (t+\ell)}[f]_{t+\ell,1}\nonumber\\
&\le \max\{1,e^\mu\}\|f\|_{1,\delta}  \sum_{\ell= 0}^\infty e^{(\mu-\delta)(t+ \ell)}\nonumber = C_{\mu,\delta}\|f\|_{1,\delta}\, e^{(\mu-\delta) t},\nonumber
\end{align}
and in the case $\mu>\delta$ we have 
\begin{align}
&\int_{0}^t e^{\mu s} |f(s)|\,ds \le  \sum_{\ell= 0}^{\lfloor t \rfloor} \int_{\ell}^{\ell+ 1}e^{\mu s} |f(s)|\,ds \le \sum_{\ell= 0}^{\lfloor t \rfloor} e^{\mu (\ell+1)}[f]_{\ell,1}\nonumber\\
&\le e^\mu \|f\|_{1,\delta}  \sum_{\ell= 0}^{\lfloor t \rfloor} e^{(\mu-\delta)\ell}= e^\mu \|f\|_{1,\delta} \frac{e^{(\mu-\delta)(\lfloor t \rfloor+1)}-1}{e^{\mu-\delta}-1} \le D_{\delta,\mu}  \|f\|_{1,\delta}\, e^{(\mu-\delta)t} \nonumber
\end{align}
with $C_{\mu,\delta}$ and $D_{\delta,\mu}$ given above.
\end{proof}

Next, for $\delta>0$, we define the function space  
$$
W^{2}_{\delta}(I)  := \left\{ u \in C_\delta^{1}(\overline I) \cap W^{2,1}_{loc}(\ov I)\:: \: u(0)= 0,\: u'' \in L^1_\delta(I) \right\}
$$
and endow this space with the norm 
$$
\|u\|_{W^{2}_\delta} := \|u\|_{C_\delta^1} + \|u''\|_{1,\delta}
$$
We first note that 
$$
u'(t)= -\int_{t}^\infty u''(s)\,ds \qquad \text{for $u \in W^{2}_\delta(I)$ and $t \ge 0$.}
$$

\begin{lemma}
\label{sec:case-2p3-2-W-banach}
$W^{2}_{\delta}(I)$ is a Banach space.  
\end{lemma}

\begin{proof}
Consider a Cauchy sequence $(u_n)_n$ in $W^2_\delta(I)$. Then we have 
\begin{equation}
  \label{eq:Banach-space-W-0}
u_n \to u \quad \text{in $C_\delta^{1}(\overline I)$}\qquad \text{and}\qquad 
u_n'' \to v \quad \text{in $L^1_\delta(I)$.}
\end{equation}
Moreover, we have 
\begin{equation}
  \label{eq:Banach-space-W}
u'(t)=\lim_{n \to \infty}u_n'(t) = -\lim_{n \to \infty}\int_{t}^\infty u_n''(s)ds=- \int_{t}^\infty v(s)\,ds \qquad \text{for all $t >0$,}
\end{equation}
since 
$$
\int_{t}^\infty |u_n''(s)-v(s)|\,ds \le C_{0,\delta}\|u'' - v\|_{1,\delta}\, 
e^{-\delta t} \to 0  \qquad \text{as $n \to \infty$}   
$$
by (\ref{eq:L-1-delta-est}). From (\ref{eq:Banach-space-W}) we deduce that $u'' = v \in L^1_\delta(I)$ in weak sense. Then it follows from (\ref{eq:Banach-space-W-0}) that $u_n \to u$ in $W^2_\delta(I)$.
\end{proof}

The following simple lemma is essential.

\begin{lemma} \label{isomorphism-p-2-3}
	Let $\delta, \gamma, \mu\ge 0$ satisfy $\delta <\sqrt{\frac{\gamma^2}{4} + \mu^2} -\frac{\gamma}{2}$. Then the map $W^{2}_\delta(I) \to L^1_\delta(I)$, $T\Psi = -\Psi'' + \gamma \Psi' + \mu^2 \Psi$ is an isomorphism.
\end{lemma}
\begin{proof}
Let $\lambda:= \sqrt{\frac{\gamma^2}{4} + \mu^2}$.
Any solution of the equation $-\Psi'' + \gamma \Psi' +\mu^2 \Psi=0$ is given by $\Psi(t)=Ae^{(\frac{\gamma}{2}-\lambda) t} + B e^{(\frac{\gamma}{2} + \lambda)  t}$ with suitable $A,B \in \R$. If $\Psi \in W^2_\delta(I)$, then $\Psi$ is bounded and therefore $B=0$. Moreover, $A=0$ since $\Psi(0)=0$, and therefore $\Psi \equiv 0$. Hence $T$ has zero kernel.

For $g \in L^1_\delta(I)$, a solution of $-\Psi'' + \gamma \Psi' + \mu^2 \Psi = g$ is given by
	$$
	\Psi(t)=\frac{1}{2 \lambda}e^{(\frac{\gamma}{2}+\lambda) t}\int_t^\infty e^{-(\frac{\gamma}{2}+\lambda) s}g(s) \, ds 
	+  \frac{1}{2 \lambda}e^{ (\frac{\gamma}{2}-\lambda) t}\int_0^t e^{(-\frac{\gamma}{2}+\lambda) s}g(s) \, ds .
	$$
 By (\ref{eq:L-1-delta-est}) and (\ref{eq:L-1-delta-est-1}), we have 
\begin{align*}
	\Bigl| e^{(\frac{\gamma}{2}+\lambda) t}\int_t^\infty e^{-(\frac{\gamma}{2}+\lambda) s}g(s) \, ds  \Bigr| &  \leq 
C_{-(\frac{\gamma}{2}+\lambda),\delta}  \|g\|_{1,\delta}\: e^{-\delta t} , \\
	\Bigl| e^{ (\frac{\gamma}{2}-\lambda) t}\int_0^t e^{(-\frac{\gamma}{2}+\lambda) s}g(s) \, ds \Bigr| & \leq 
D_{-\frac{\gamma}{2}+\lambda,\delta}  \|g\|_{1,\delta}\: e^{-\delta t} 
\end{align*}
for $t \ge 0$. Hence $\Psi \in C_\delta(I)$. Since 
        \begin{equation}
          \label{eq:psi-prim-rep}
        \Psi'(t)= \frac{\frac{\gamma}{2}+\lambda}{2 \lambda}e^{(\frac{\gamma}{2}+\lambda) t}\int_t^\infty e^{-(\frac{\gamma}{2}+\lambda) s}g(s) \, ds  
        + \frac{\frac{\gamma}{2}-\lambda}{2 \lambda}e^{ (\frac{\gamma}{2}-\lambda) t}\int_0^t e^{(-\frac{\gamma}{2}+\lambda) s}g(s) \, ds
	\end{equation}
it also follows that $\Psi' \in C_\delta(I)$. Additionally, we have 
$\Psi''=\mu^2 \Psi + \gamma \Psi' -g \in L^1_\delta$. By adding a multiple of the function $t \mapsto e^{(\frac{\gamma}{2}-\lambda) t}$, we can ensure that $\Psi(0)=0$ and therefore $\Psi \in W^2_\delta(I)$. We conclude that $T$ is an isomorphism.	
\end{proof}

From now on, we fix $\gamma_\sdiamond \in (0,\frac{N-2}{N+\alpha_p})$, By Proposition~\ref{sec:preliminaries-limit-problem-3} and Lemma~\ref{eq:gamma-sufficient}, we have
\begin{equation}
  \label{eq:def-gamma-max}
\nu_\sdiamond:= \sup_{0 \le \gamma \le \gamma_\sdiamond}\nu_K(\gamma)<0.  
\end{equation}
Moreover, we fix  
\begin{equation}
  \label{eq:fix-delta-spectral}
\delta := \min \left\{\frac{\sqrt{1-2\nu_\sdiamond}-1}{2}\:,\:\frac{1}{2}\Bigl(\sqrt{\frac{\gamma_\sdiamond^2}{4} - \nu_\sdiamond} -\frac{\gamma_\sdiamond}{2}\Bigr)\:,\:\frac{2}{N}\right\} 
\end{equation}
for the remainder of this section. By Lemma~\ref{bounded-sol-exp-decay} and since $\delta \le \frac{1}{2}\bigl(\sqrt{1-2\nu_\sdiamond}-1\bigr)$, there exists $C>0$ such that 
\begin{equation}
  \label{eq:fix-delta-spectral-consequence}
\|\Psi\|_{C^2_\delta(I)} \le C \|\Psi\|_{L^\infty(I)} 
\end{equation}
for every eigenfunction of (\ref{eq:trans-weighted-eigenvalue-problem-rewritten}) corresponding to $\gamma \in [0,\gamma_\sdiamond]$ and $\nu = \nu_j(\gamma)$, $j=1,\dots,k$.   

We consider the spaces $E_\delta:= W^{2}_\delta(I)$ and $F_\delta := L^1_\delta(I)$. The key observation of this section is the following. 
\begin{proposition}
\label{G-differentiable-p-2-3}
Let $\eps_0>0$ be given by Proposition~\ref{implicit function for U-gamma}, so that $(-\eps_0,\gamma_\sdiamond) \to C^1_0(I)$, $\gamma \mapsto U_\gamma$ is a well defined $C^1$-map by Remark~\ref{gamma-negative-definition}. Moreover, let the map 
  $$
G: \left(-\eps_0,\gamma_\sdiamond \right) \times {E_\delta}  \times \R  \to {F_\delta} \times \R 
$$
be defined by (\ref{def-G-implicit}). Then $G$ is of class $C^1$ with 
\begin{align*} 
&\partial_{\gamma}G(\gamma,\Psi,\nu) = 
\begin{pmatrix}
\Psi'   - (p-1)e^{(\gamma-1)t}|U_\gamma|^{p-2} \Bigl(t + (p-2)\frac{U_\gamma \partial_\gamma U_\gamma}{|U_\gamma|^2}\Bigr)\Psi\\
0
\end{pmatrix},
\quad \partial_{\nu}G(\gamma,\Psi,\nu)= 
\begin{pmatrix}
 -\Psi \\
0
\end{pmatrix}\\
&\text{and}\qquad d_{\Psi}G(\gamma,\Psi,\nu)\phi = 
\begin{pmatrix}
-\phi''  +\gamma \phi' - (p-1)e^{(\gamma-1)t}|U_\gamma|^{p-2}\phi - \nu \phi \\
\int_0^\infty \Psi \phi \, dt
\end{pmatrix} \quad \text{in $F_\delta \times \R$}
\end{align*}
for $\phi \in E_\delta$. 
\end{proposition}

We postpone the somewhat lengthy proof of this proposition to the next section and continue the main argument first. We fix $j \in \{1, \ldots, K\}$ and for $\gamma \ge 0$ we let $\Psi_{\gamma,j}$ denote an eigenfunction of the eigenvalue problem \eqref{eq:trans-weighted-eigenvalue-problem-rewritten} corresponding to the eigenvalue $\nu_j(\gamma)$. We thus have  
$$
-\Psi_{\gamma,j}'' + \gamma \Psi_{\gamma,j}' - (p-1)e^{(\gamma-1)t}|U_\gamma(t)|^{p-2}\Psi_{\gamma,j}= 
\nu_j(\gamma) \Psi_{\gamma,j} \; \text{in $[0,\infty)$,} \quad \Psi_{\gamma,j}(0)=0,\: \Psi_{\gamma,j} \in L^\infty(I).
$$
By (\ref{eq:fix-delta-spectral-consequence}) we have $\Psi_{\gamma,j} \in {E_\delta}$. Moreover, we can assume $\int_0^\infty \Psi_{\gamma,j}^2 \, dt=1$ so that 
$$
G(\gamma,\Psi_{\gamma,j},\nu_j(\gamma))=0.
$$ 

To apply the implicit function theorem to $G$ at the point $(\gamma,\Psi_{\gamma,j},\nu_j(\gamma))$, we need the following property.

\begin{proposition}
\label{deriv-isomorphism}
	Let $\gamma \in [0,\gamma_\sdiamond]$. Then the map 
\begin{align*} 
& L:=d_{\Psi,\nu}G(\gamma,\Psi_{\gamma,j},\nu_j(\gamma)): {E_\delta} \times \R  \to {F_\delta} \times \R \\ 
& (\phi,\rho)  \mapsto  \begin{pmatrix}
-\phi''  + \gamma \phi' - (p-1)e^{(\gamma-1)t}|U_\gamma|^{p-2}\phi - \nu_j(\gamma) \phi - \rho \Psi_{\gamma,j} \\
\int_0^\infty \Psi_{\gamma,j} \phi \, dt
\end{pmatrix}
\end{align*}
is an isomorphism.
\end{proposition}

\begin{proof}
	Since, by definition, 
$$
\delta < \sqrt{\frac{\gamma_\sdiamond^2}{4} -\nu_\sdiamond} -\frac{\gamma_\sdiamond}{2} \le \sqrt{\frac{\gamma^2}{4} -\nu_j(\gamma)} -\frac{\gamma}{2},
$$
we may apply Lemma~\ref{isomorphism-p-2-3} with $\mu = \sqrt{-\nu_j(\gamma)}$. Hence the map ${E_\delta} \to {F_\delta}$, $\phi \mapsto -\phi'' + \gamma \phi' - \nu_j(\gamma) \phi$ is an isomorphism. Since the linear map ${E_\delta} \to {F_\delta}$, $\phi \mapsto (p-1)e^{(\gamma-1)t}|U_\gamma|^{p-2}\phi$ is compact, the map 
	\begin{align*} 
	T: {E_\delta} &\to {F_\delta} \\
	\phi & \mapsto -\phi'' + \gamma \phi'  - (p-1)e^{(\gamma-1)t}|U_\gamma|^{p-2}\phi - \nu_j(\gamma) \phi
	\end{align*}
	is a Fredholm operator of index zero. The kernel of this map is one dimensional, since it consists of eigenfunctions corresponding to $\nu_j(\gamma)$. 
	Hence the codimension of the image of $T$ is one, and we claim that $\Psi_{\gamma,j}$ is not contained in the image of $T$. Otherwise, there exists $\phi \in {E_\delta}$ such that 
	$-\phi'' + \gamma \phi'  - (p-1)e^{(\gamma-1)t}|U_\gamma|^{p-2}\phi - \nu_j(\gamma) \phi=\Psi_{\gamma,j}$. Multiplying with $\Psi_{\gamma,j}$ and integrating by parts then yields
        \begin{align*}
        0 < \int_0^\infty e^{-\gamma t}\Psi_{\gamma,j}^2 \, dt 
        &= \int_0^\infty (-(e^{-\gamma t}\phi')' - (p-1)e^{-t}|U_\gamma(t)|^{p-2}\phi- \nu_j(\gamma) e^{-\gamma t} \phi)\Psi_{\gamma,j} \, dt \\
&= \int_0^\infty(-(e^{-\gamma t}\Psi')' - (p-1)e^{-t}|U_\gamma(t)|^{p-2}\Psi 
 - \nu_j(\gamma) e^{-\gamma t}\Psi) \phi \, dt =0,
        \end{align*}
	a contradiction. It follows that 
	\begin{equation} \label{decomposition}
	{E_\delta}=\textrm{range}\, T \oplus \text{span} \{\Psi_{\gamma,j}\} .
	\end{equation}	
	We now show that $L$ is an isomorphism. First assume $L(\phi,\rho)=0$ for some $(\phi,\rho) \in {E_\delta} \times \R$, i.e., 
	$$
	-\phi'' + \gamma \phi'  - (p-1)e^{(\gamma-1)t}|U_\gamma|^{p-2}\phi - \nu_j(\gamma) \phi = \rho\Psi_{\gamma,j} \quad \text{in ${F_\delta}$} \qquad \text{and}\qquad \int_0^\infty\Psi_{\gamma,j} \phi \, dt =0.
	$$
	Since $\Psi_{\gamma,j} \not \in \textrm{range}\, T $, the first equality yields $\rho=0$. But then $\phi$ itself is an eigenfunction and therefore $\phi = c\Psi_{\gamma,j}$ for some $c \in \R$. The second equality then yields $c=0$, and thus $(\phi,\rho)=(0,0)$. Hence $L$ is injective. 
	
	Now let $(g, \sigma) \in {F_\delta} \times \R$. By \eqref{decomposition} there exist $g_0 \in \textrm{range}\, T$, $\kappa \in \R$ such that $g=g_0 + \kappa\Psi_{\gamma,j}$. Since $g_0 \in \textrm{range}\, T$, there exists a solution $\phi_0 \in {E_\delta}$ of  
	$$
	-\phi'' + \gamma \phi'  - (p-1)e^{(\gamma-1)t}|U_\gamma|^{p-2}\phi - \nu_j(\gamma) \phi = g_0  \qquad \text{in $I$.}
	$$
	 Furthermore, for any $\eta \in \R$, $\phi_0 + \eta \Psi_{\gamma,j} \in {E_\delta}$ is also a solution. Taking $\eta=\sigma-\int_0^\infty\Psi_{\gamma,j} \phi_0 \,dt$ yields
	$$
	\int_0^\infty\Psi_{\gamma,j} (\phi_0+\eta\Psi_{\gamma,j}) \, dt=\sigma .
	$$
	Consequently, we have
	$$
	L(\phi_0+\eta\Psi_{\gamma,j}, -\kappa) = {g \choose \sigma}.
	$$
	Hence $L$ is surjective. 
\end{proof}

With the help of Propositions \ref{G-differentiable-p-2-3} and \ref{deriv-isomorphism}, we may now apply the implicit function theorem to $G$ at 
$(\gamma,\Psi_{\gamma,j},\nu_j(\gamma))$. This yields the following result.  
 
\begin{corollary} \label{implicit function-spectral-asymptotics}
	There exist $\eps_1 \in (0,\eps_0)$ and, for $j=1,\dots,K$, $C^1$-maps $h_j: (-\eps_1,\gamma_\sdiamond) \to \R$ with 
the property that 
\begin{equation}
  \label{eq:g-j-2-equality}
h_j(\gamma) = \nu_j(\gamma) \qquad \text{for $j = 1,\dots,K$, $\gamma \in [0,\gamma_\sdiamond)$}
\end{equation}
and 
\begin{equation}
  \label{eq:limit-der-expression-gamma}
h_j'(0)=  -(p-1) \int_0^\infty \left( t e^{-t} |U_0|^{p-2}\Psi_{0,j}^2 + (p-2) e^{-t} |U_0|^{p-4} U_0 (\del_\gamma \big|_{\gamma=0}\, U_\gamma)\Psi_{0,j}^2 \right) \, dt 
\end{equation}
 for $j=1,\dots,K$. 
\end{corollary}

\begin{proof}
By Propositions \ref{G-differentiable-p-2-3}, \ref{deriv-isomorphism} and the implicit function theorem applied to the map $G$ at $(0,\Psi_{0,j},\nu_j(0))$, there exists $\eps_1 \in (0,\eps_0)$ and $C^1$-maps $g_j: (-\eps_1,\eps_1) \to F_\delta \times \R$ with the property that $g_j(0)=(\Psi_{0,j},\nu_j(0))$ and 
$G(\gamma,g_j(\gamma))=0$ for $\gamma \in (-\eps_1,\eps_1)$. Let $h_j$ denote the second component of $g_j$. Since 
$$
\nu_1(0) = h_1(0) < \nu_2(0) = h_2(0) < \dots < \nu_K(0) = h_K(0)< 0, 
$$
we may, after making $\eps_1$ smaller if necessary, assume that also 
$$
h_1(\gamma) < h_2(\gamma) < \dots < h_K(\gamma)<0 \qquad \text{for $\gamma \in (0,\eps_1)$.} 
$$
Since, by construction, the values $h_j(\gamma)$ are eigenvalues of (\ref{eq:trans-weighted-eigenvalue-problem}) and the negative eigenvalues of (\ref{eq:trans-weighted-eigenvalue-problem}) are precisely given by (\ref{var-char-nu-j}), the equality (\ref{eq:g-j-2-equality}) follows for $\gamma \in (0,\eps_1)$. Using Propositions \ref{G-differentiable-p-2-3}, \ref{deriv-isomorphism} and applying the implicit function theorem at $(\gamma,\Psi_{\gamma,j},\nu_j(\gamma))$, the functions $h_j$ may be extended as $C^1$-functions to $(-\eps_1,\gamma_\sdiamond)$ such that (\ref{eq:g-j-2-equality}) holds for $(0,\gamma_\sdiamond)$.
Moreover, (\ref{eq:limit-der-expression-gamma}) is a consequence of implicit differentiation of the equation 
$G(\gamma,g_j(\gamma))=0$.
\end{proof}

We may now complete the 

\begin{proof}[Proof of Theorem~\ref{spectral-curves}] We first note that -- since $U_0:=(-1)^{K-1}U_\infty$ -- the eigenvalue problem (\ref{eq:weighted-eigenvalue-translimit-preliminaries}) coincides with (\ref{eq:weighted-eigenvalue-translimit}), and it has precisely $K$ negative eigenvalues $\nu_j^*:= \nu_j(0)$, $j=1,\dots,K$ by Proposition~\ref{sec:preliminaries-limit-problem-3}. To prove the expansions~(\ref{expansions}), we fix $j \in \{1,\dots,K\}$. By Remark \ref{introduction-remark-1} and Corollary \ref{implicit function-spectral-asymptotics}, the constant $c_j^*$ appearing in \eqref{expansions} is given by $c_j^* = 2 N \nu_j^* + (N-2) h_j'(0).$ Now  Corollary \ref{implicit function-spectral-asymptotics} yields the expansions
\begin{equation} \label{eigenvalue: Taylor expansion}
\nu_j(\gamma) = \nu_j^* + \gamma h_j'(0) + o(\gamma)\qquad \text{and}\qquad 
\partial_\gamma \nu_j(\gamma) = h_j'(0) + o(1) \qquad \text{as $\gamma \to 0^+$.}
\end{equation} 
Writing $\gamma= \gamma(\alpha)= \frac{N-2}{N+\alpha}$ as before and recalling  
\eqref{eq:transformed-variables}, we thus have
\begin{align*} 
\mu_j(\alpha) & = (N+\alpha)^2 \nu_j(\gamma(\alpha)) = (N+\alpha)^2 \left(
\nu_j^* + \frac{N-2}{N+\alpha} h_j'(0) + o\left(\frac{1}{\alpha}\right)\right)\\
& = \nu_j^* \,\alpha^2  +  \bigl[2N \nu_j^* + (N-2) h_j'(0)\bigr]\alpha + o(\alpha) =  \nu_j^* \,\alpha^2 +  c_j^*\,\alpha + o(\alpha)
\end{align*}
and 
\begin{align*}
\mu_j'(\alpha) & = 2(N+\alpha) \nu_j(\gamma(\alpha)) - (N-2) [\del_\gamma \nu_j](\gamma(\alpha)) \\
& = 2(N+\alpha) \left(\nu_j^* + \frac{N-2}{N+\alpha} h_j'(0) + o\left(\frac{1}{\alpha} \right) \right) - (N-2)(h_j'(0) + o(1)) \\
& = 2 \nu_j^* \,\alpha    + 2 N \nu_j^* + (N-2) h_j'(0) + o(1) = 2 \nu_j^*\, \alpha + c_j^* + o(1)\quad \text{as $\alpha \to \infty$.}
\end{align*}
\end{proof}

We may also complete the 

\begin{proof}[Proof of Theorem~\ref{corollary-on-eigenvalue-curves}]
By Theorem~\ref{spectral-curves} we have 
$$
\mu_{i}'(\alpha) = 2 \alpha \nu^*_i + c^*_i +o(1)  \qquad \text{as $\alpha \to \infty$}
$$
for $i=1,\dots,K$. Since the values $\nu^*_i$ are negative, we may thus fix $\alpha_*> 0$ such that 
\begin{equation}
  \label{eq:strict-deriv-ineq}
\mu_{i}'(\alpha)< 0 \qquad \text{for $\alpha \ge \alpha_*$, $i=1,\dots,K$.}
\end{equation}
We now fix $i \in \{1,\dots,K\}$. Then there exists a minimal positive integer $\ell_i$ 
such that 
$$
\mu_{i}(\alpha_*)+ \lambda_\ell > 0 \qquad \text{for $\ell \ge \ell_i$.}
$$
Moreover, since $\mu_{i}(\alpha) \to -\infty$ as $\alpha \to \infty$ by Theorem~\ref{spectral-curves}, there exists, for every $\ell \ge \ell_i$,  precisely one value $\alpha_{i,\ell} \in (\alpha_*,\infty)$ such that 
$$
\mu_{i}(\alpha_{i,\ell})+ \lambda_\ell = 0.
$$
Fix such a value $\alpha_{i,\ell}$ and put $\delta_{i,\ell}= \alpha_{i,\ell}-\alpha_*$. Since the curves $\alpha \mapsto \mu_j(\alpha)$, $j=1,\dots,K$ are bounded on the interval $[\alpha_*,\alpha_{i,\ell}+\delta_{i,\ell}]$, it follows that the set
$$
N_{i,\ell}:= \left \{
  \begin{aligned}
  &(j,\ell') \in \{1,\dots,K\} \times (\N \cup \{0\}) \::\\
  &\mu_{j}(\alpha)+ \lambda_{\ell'}= 0 \; \text{for some $\alpha \in [\alpha_*,\alpha_{i,\ell}+\delta_{i,\ell}]$}
  \end{aligned}
\right \}
$$
is finite. Combining this fact with (\ref{eq:strict-deriv-ineq}), we find $\eps_{i,\ell} \in (0,\delta_{i,\ell})$ such that 
$$
\mu_{j}(\alpha)+\lambda_{\ell'} \not = 0 \qquad \text{for $\alpha \in (\alpha_{i,\ell}-\eps_{i,\ell},\alpha_{i,\ell}+\eps_{i,\ell}) \setminus \{\alpha_{i,\ell}\}$, $j=1,\dots,K$ and $\ell' \in \N \cup \{0\}$.}
$$
From Proposition~\ref{spectral-curves-0}, it then follows that $u_\alpha$ is nondegenerate for $\alpha \in (\alpha_{i,\ell}-\eps_{i,\ell},\alpha_{i,\ell}+\eps_{i,\ell})$, $\alpha \not = \alpha_{i,\ell}$. Finally, it also follows from Proposition~\ref{spectral-curves-0} and  (\ref{eq:strict-deriv-ineq}) that 
$$
m(u_{\alpha_{i,\ell}+\eps})-m(u_{\alpha_{i,\ell}-\eps}) = \sum_{(j,\ell') \in M_{i,\ell}} d_{\ell'}
>0 \qquad \text{for $\eps \in (0,\eps_{i,\ell})$,}
$$ 
where $M_{i,\ell} \subset \{1,\dots,K\} \times (\N \cup \{0\}$ is the set of pairs $(j,\ell')$ with $\mu_{j}(\alpha_{i,\ell})+ \lambda_{\ell'}=  0$ and, as before, $d_{\ell'}$ is the dimension of the space of spherical harmonics of degree $\ell'$. Here we note that $M_{i,\ell} \not = \varnothing$ since it contains $(i,\ell)$. 
\end{proof}

\section{Differentiability of the map $G$}
\label{sec:differentiability-g}

In this section, we give the proof of Proposition~\ref{G-differentiable-p-2-3}, which we restate here in a slightly more general form. As before, we fix $p>2$ and $\gamma_\sdiamond \in [0,\frac{N-2}{N+\alpha_p})$.
\begin{proposition}
\label{G-differentiable-p-2-3-restated}
Let $\eps_0 \in (0,\frac{1}{2})$ be given by Proposition~\ref{implicit function for U-gamma}, so that the map $(-\eps_0,\gamma_\sdiamond) \to C^1_0(I)$, $\gamma \mapsto U_\gamma$ is well defined and differentiable by Remark~\ref{gamma-negative-definition}. Let, furthermore, $\delta \in (0,\frac{2}{N})$, and let the map 
  $$
G: \left(-\eps_0,\gamma_\sdiamond \right) \times W^2_\delta(I)  \times \R  \to L^1_\delta(I) \times \R $$
be defined by (\ref{def-G-implicit}). Then $G$ is of class $C^1$ with 
\begin{align*} 
&d_{\gamma}G(\gamma,\Psi,\nu) = 
\begin{pmatrix}
\Psi'   - (p-1)e^{(\gamma-1)t}|U_\gamma|^{p-2} \Bigl(t + (p-2)\frac{U_\gamma \partial_\gamma U_\gamma}{|U_\gamma|^2}\Bigr)\Psi\\
0
\end{pmatrix},
\quad &d_{\nu}G(\gamma,\Psi,\nu)= 
\begin{pmatrix}
 -\Psi \\
0
\end{pmatrix}\\
&\text{and}\qquad d_{\Psi}G(\gamma,\Psi,\nu)\phi = 
\begin{pmatrix}
-\phi''  +\gamma \phi' - (p-1)e^{(\gamma-1)t}|U_\gamma|^{p-2}\phi - \nu \phi \\
\int_0^\infty \Psi_0 \phi \, dt
\end{pmatrix}.
\end{align*}
\end{proposition}

The remainder of this section is devoted to the proof of this proposition. We first note that, by Lemma~\ref{bounded-sol-E}, $U_\gamma$ has a finite number of simple zeros and satisfies $\lim \limits_{t \to \infty}|U_\gamma(t)|>0$ for $\gamma \in \left(-\eps_0, \gamma_\sdiamond \right)$. The key step in the proof of Proposition~\ref{G-differentiable-p-2-3-restated} is the following lemma. 

\begin{lemma}
\label{continuity-2-p-3}
Let $q >0$, and let $\cU \subset C^1_0(I)$ be the open subset of functions $u \in C^1_0(I)$ which have a finite number of simple zeros and satisfy $\lim \limits_{t \to \infty}|u(t)|>0$. Then the nonlinear map 
$$
h_q: \cU \to L^1_0(I),\qquad u \mapsto |u|^q
$$
is of class $C^1$ with 
$$
h_q'(u) w = q|u|^{q-2}u w \;\in\; L^1_0(I)\qquad \text{for 
$u \in \cU, w \in C^1_0(I)$.}
$$ 
Here we identify $|u|^{q-2}u$ with $\sgn(u)$ in the case $q=1$.
\end{lemma}

\begin{proof}
We only consider the case $q \in (0,1)$. The proof in the case $q=1$ is similar 
but simpler, and the proof in the case $q>1$ is standard. We first prove\\
{\em \underline{Claim 1:}} If $1 \le r < \frac{1}{1-q}$, then the map $\sigma_q : \cU \to L^r_0(I),\: \sigma_q(u)=|u|^{q-2}u$ is well defined and continuous.\\
To see this, we note that, by definition of $\cU$, for every $u \in \cU$ we have\begin{equation}
  \label{parini-w-est-1}
\kappa_u:= \sup\left\{ \frac{|\{|u| \le \tau\} \cap (t,t+1)|}{\tau}: \tau>0, \ t \geq 0 \right\} <\infty.
\end{equation}
More generally, if $K \subset \cU$ is a compact subset (with respect to $\|\cdot\|_{C^1_0}$), we also have that 
$$
\kappa_{\text{\tiny $K$}}:= \sup_{u \in K}\kappa_u < \infty.
$$
As a consequence of (\ref{parini-w-est-1}), we have  
\begin{align*}
\int_{t}^{t+1}|\sigma_q(u)|^{r}\,dx &= \int_{t}^{t+1}|u|^{(q-1)r}\,dx
=\int_0^\infty |(t,t+1) \cap \{|u|^{(q-1)r} \ge s\}|\,ds \\
&=
\int_0^\infty |(t,t+1) \cap \{|u| \le s^{\frac{1}{(q-1)r}}\}|\,ds \le 
\int_0^{\infty} \min  \{1, \kappa_u\, s^{\frac{1}{(q-1)r}}\}\,ds < \infty
\end{align*}
for every $u \in \cU$ and $t \ge 0$, since $\frac{1}{(q-1)r}<-1$ by assumption. Hence $\sigma_q(u) \in L^r_0(I)$ for every $u \in \cU$, so the map $\sigma_q$ is well defined.
To see the continuity of $\sigma_q$, let $(u_n)_n
\subset \cU$ be a sequence such that $u_n \to u \in \cU$ as $n \to \infty$
with respect to the $C^1_0$-norm. We then consider the compact set $K:=
\{u_n,u \::\: n \in \N\}$. For given $\eps>0$, we fix $c \in (0,1)$ sufficiently small such that 
\begin{equation}
  \label{eq:parini-weth-9}
c^{(q-1)r+1} < \frac{\eps}{2^{r} \kappa_{\text{\tiny $K$}} \Bigl(\frac{ 2^{1+(q-1)r}}{1+(q-1)r} \Bigr)} .
\end{equation}
Since $u_n \to u$ uniformly on $[0,\infty)$, it is easy to see
that 
\begin{equation}
  \label{eq:11}
\sup_{t \ge 0} \int_{t}^{t+1} 1_{\{|u| > c\}} \bigl|\sigma_q(u_n) -\sigma_q(u)\bigr|^r \,dx \to 0 \qquad
\text{as $n \to \infty$.}
\end{equation}
Moreover, there exists $n_0 \in \N$ with the property that 
$$
\{|u| \le c \} \subset
\{|u_n| \le 2c\} \qquad \text{for $n \ge n_0$.}
$$
Consequently, setting $v_n:= |u_n|^{(q-1)r}$ for $n \ge n_0$ and $v:= |u|^{(q-1)r}$, we find that  
\begin{align*}
&\sup_{t \ge 0} \int_{t}^{t+1} 1_{\{|u| \le c\} } \Bigl|\sigma_q(u_n)-\sigma_q(u)\Bigr|^r \,dx  \le 
2^{r-1} \int_{\{|u| \le c\} \cap (t,t+1)}\Bigl(|u_n|^{(q-1)r}+ |u|^{(q-1)r}\Bigr) \,dx\\
&\le 
2^{r-1} \Bigl(\int_{\{|u_n| \le 2 c\} \cap (t,t+1)} |u_n|^{(q-1)r} \,dx +   \int_{\{|u| \le c\} \cap (t,t+1)} 
|u|^{(q-1)r}\,dx \Bigr)\\
&= 
2^{r-1} \Bigl(\int_{\{v_n \ge (2 c)^{(q-1)r}\} \cap (t,t+1)} v_n \,dx +   
\int_{\{v \ge c^{(q-1)r}\} \cap (t,t+1)}v \,dx \Bigr)\\ 
&= 2^{r-1} \Bigl( \int_{(2c)^{(q-1)r}}^\infty
|\{v_n \ge s\} \cap (t,t+1)|\,ds + (2c)^{(q-1)r} |\{v_n \ge (2c)^{(q-1)r}\} \cap (t,t+1)|\\
&+   \int_{c^{(q-1)r}}^\infty |\{v \ge s\} \cap (t,t+1)|\,ds +c^{(q-1)r} |\{v_n \ge c^{(q-1)r}\} \cap (t,t+1)|\Bigr) \\
&= 2^{r-1} \Bigl( \int_{(2c)^{(q-1)r}}^\infty
|\{|u_n| \le s^{\frac{1}{(q-1)r}}\} \cap (t,t+1)|\,ds + (2c)^{(q-1)r} |\{|u_n| \le 2c\} \cap (t,t+1)|\\
&+   \int_{c^{(q-1)r}}^\infty |\{|u| \le s^{\frac{1}{(q-1)r}}\} \cap (t,t+1)|\,ds +
c^{(q-1)r} |\{|u| \le c\} \cap (t,t+1)|\Bigr)\\
&\le 2^{r} \kappa_{\text{\tiny $K$}} \Bigl( \int_{(2c)^{(q-1)r}}^\infty s^{\frac{1}{(q-1)r}}ds 
  + (2c)^{1+(q-1)r}\Bigr)\\
&= 2^{r} \kappa_{\text{\tiny $K$}} \Bigl( - \frac{(2c)^{(q-1)r+1}}{\frac{1}{(q-1)r}+1}
  + (2c)^{1+(q-1)r}\Bigr)\\
&= 2^{r} \kappa_{\text{\tiny $K$}} \Bigl(\frac{ 2^{1+(q-1)r}}{1+(q-1)r} \Bigr)c^{(q-1)r+1} < \eps \qquad \text{for $n \ge n_0$}
\end{align*}
by (\ref{eq:parini-weth-9}). Combining this with (\ref{eq:11}) yields
$$
\limsup_{n \to \infty} \|\sigma_q(u_n)-\sigma_q(u)\|_{r,0}^r =  \limsup_{n \to \infty} \sup_{t \ge 0} [\sigma_q(u_n)-
\sigma_q(u)]_{t,r}^r \le \eps.
$$ 
Since $\eps>0$ was given arbitrarily, we conclude that 
$$
\|\sigma_q(u_n)-\sigma_q(u)\|_{r,0}^r  \to 0 \qquad \text{as $n \to \infty$.}
$$
Hence Claim 1 follows.\\
Next, we let $u \in \cU$ and $w \in C^1_0(I)$ with
$\|w\|_{L^\infty(I)}<1$. For $\tau \in \R
\setminus \{0\}$ we then have 
$$
\frac{1}{\tau}\Bigl(h_q(u+\tau w) -h_q(u)\Bigr)=I_{{\tau}}+J_{{\tau}}\quad \text{in $L^1_0(I)$}
$$
with
$$
I_{{\tau}}(x)= 1_{\{|u|> |{\tau}|\} }\frac{|u+{\tau}w|^{q} -|u(x)|^{q}}{{\tau}}
,\quad  J_{{\tau}}= 1_{\{|u|\le |{\tau}|\} }\frac{|u+{\tau}w|^{q} -|u|^{q}}{{\tau}}
$$ 
Note that  
$$
I_{{\tau}}(x)= q  \int_0^1 1_{\{|u|> |{\tau}|\} }(x)\sigma_q(u(x)+\rho{\tau}w(x))w(x)
\,d\rho.
$$
Hence 
$$
\bigl[I_{\tau} - q \sigma_q(u) w\bigr](x) = q \int_{0}^1 \Bigl[\sigma_q(u+\rho  {\tau} w)w - \sigma_q(u)w\Bigr](x)\,d\rho - q  \int_0^1 \Bigl[1_{\{|u| \le |{\tau}|\} }\sigma_q(u+\rho  {\tau} w)w\Bigr](x) d\rho
$$
where
$$
\int_{t}^{t+1} \Bigl| \int_{0}^1 \Bigl[\sigma_q(u+\rho  {\tau} w)w - \sigma_q(u)w \Bigr](x)\,d\rho\Bigr|dx \le \|w\|_{L^\infty(I)} \sup_{0 \le \rho \le 1} \| \sigma_q(u+\rho\tau w)-\sigma_q(u)\|_{1,0} \quad \text{for $t \ge 0$}
$$
and, by H\"older's and Jensen's inequality,
\begin{align*}
\int_{t}^{t+1} &\Bigl|\Bigl[ 1_{\{|u| \le |{\tau}|\} } \int_0^1 \sigma_q(u+\rho  {\tau} w)wd\rho \Bigr](x)\Bigr| dx \\
&\le |\{|u| \le {\tau}\} \cap (t,t+1)|^{1/r'}\|w\|_{L^\infty(I)} 
 \Bigl(\int_0^1 \int_{t}^{t+1} |\sigma_q(u+\rho\tau w)|^r  
 dx d\rho\Bigr)^{1/r}\\
&\le |\{|u| \le {\tau}\}|^{1/r'}\|w\|_{L^\infty(I)}  
\sup_{0 \le \rho \le 1} \| \sigma_q(u+\rho\tau w)\|_{r,0} \qquad \text{for $t \ge 0$.}
\end{align*}
Combining these two estimates with Claim 1 and (\ref{parini-w-est-1}), we deduce that 
\begin{equation}
\label{I-tau-est}
\|I_{\tau} - q \sigma_q(u) w\|_{1,0} \to 0 \qquad \text{as $\tau \to 0$.}
\end{equation}
Next we estimate  
\begin{align*}
\int_{t}^{t+1} |J_{{\tau}}| dx &\le \frac{1}{|{\tau}|} \int_{t}^{t+1} 1_{\{|u|\le |{\tau}|\}} \Bigl(|u+{\tau}w|^{q}
+|u|^{q}\Bigr)\,dx\\
&=|{\tau}|^{q-1} \int_{t}^{t+1} 1_{\{|u|\le |{\tau}|\}}
\Bigl|\frac{u}{{\tau}}+w\Bigr|^{q} +\Bigl|\frac{u}{{\tau}}\Bigr|^{q}\,dx\\
&\le |{\tau}|^{q-1}(2^{q}+1)|\{u| \le |{\tau}|\}\cap(t,t+1)| \le {\kappa_{\text{\tiny $K$}}} |{\tau}|^{q}(2^{q}+1)\qquad \text{for $t \ge 0$}
\end{align*}
and therefore 
\begin{equation}
\label{J-tau-est}
\|J_{\tau}\|_{1,0} \to 0 \qquad \text{as $\tau \to 0$.}
\end{equation}
Combining (\ref{I-tau-est}) and (\ref{J-tau-est}), we deduce the existence of 
$$
h_q'(u)w = 
\lim_{{\tau} \to \infty}\frac{1}{{\tau}}\Bigl(h_q(u+{\tau}w) -h_q(u)\Bigr)=
\sigma_q(u)w \qquad \text{in $L^1_0(I)$.}
$$
Together with Claim 1, this yields that $h_q$ is of class $C^1$, as claimed.
\end{proof}

We may now complete the 
\begin{proof}[Proof of Proposition~\ref{G-differentiable-p-2-3-restated}]
The $C^1$-regularity of $G$ follows easily once we have seen that the map 
$$
H: \left(-\eps_0, \gamma_\sdiamond \right) \times W^{2}_\delta(I) \to L^1_\delta(I), \qquad  
(\gamma, \Psi) \mapsto  e^{(\gamma-1)t}|U_{\gamma}|^{p-2}\Psi 
$$
is of class $C^1$. Note that we can write $H= H_3 \circ H_2 \circ H_1$ with 
\begin{align*}
H_1&: \left(-\eps_0, \gamma_\sdiamond \right) \times W^{2}_\delta(I) \to 
\left(-\eps_0, \gamma_\sdiamond \right) \times L^\infty(I) \times  C^1_0(I), \qquad  (\gamma, \Psi) \mapsto  (\gamma,\Psi,U_{\gamma})\\ 
H_2&: \left(-\eps_0, \gamma_\sdiamond \right) \times L^\infty(I) \times  \cU \to 
\left(-\eps_0, \gamma_\sdiamond \right) \times L^\infty(I) \times  L^1_0(I), \qquad  (\gamma, \Psi,v) \mapsto  (\gamma,\Psi,|v|^{p-2})\\ 
H_3&: \left(-\eps_0, \gamma_\sdiamond \right) \times L^\infty(I) \times  L^1_0(I) \to L^1_\delta(I), \qquad (\gamma,\psi,v) \mapsto e^{(\gamma-1)(\cdot)}v \psi  
\end{align*}
The $C^1$-regularity of $H_1$ is a consequence of Proposition~\ref{implicit function for U-gamma}, and the $C^1$-regularity of $H_2$ is a consequence of Lemma~\ref{continuity-2-p-3}. Finally, the $C^1$-regularity of $H_3$ is easy to check since $e^{(\gamma-1)t} \le e^{-\delta t}$ for $\gamma <\gamma_\sdiamond$. Hence we conclude that $H$ is of class $C^1$, and this finishes the proof.  
\end{proof}

\section{Bifurcation of almost radial nodal solutions}
\label{sec:bifurc-almost-radi}

In this section, we prove the bifurcation result stated in Theorem~\ref{thm-bifurcation}.

\begin{proof}[Proof of Theorem~\ref{thm-bifurcation}]
The proof relies on Corollary~\ref{corollary-on-eigenvalue-curves} and a result by Kielh{\"o}fer \cite{kielhoefer:1988}. To adapt our problem to the setting of \cite{kielhoefer:1988}, we consider the Hilbert space $E:=L^2(\B)$, $D:=H^2(\B) \cap H_0^1(\B)$, fix $\alpha:=\alpha_{i,\ell}$ as in the assumption and consider the map
$$
G: (-\alpha,\infty) \times D \to E, \quad [G(\lambda,u)] = -\Delta (u+u_{\alpha + \lambda}) - |x|^{\alpha + \lambda} |u+u_{\alpha + \lambda}|^{p-2} (u+u_{\alpha + \lambda}) .
$$
Then $G$ is continuous with $G(\lambda,0)=0$ for $\lambda>-\alpha$. Moreover, the Fr\'echet derivative $A(\lambda):=G_u(\lambda,0)$, given by
$$
A(\lambda) \phi = -\Delta \phi - (p-1) |x|^{\alpha + \lambda} |u_{\alpha + \lambda}|^{p-2} \phi ,
$$
exists for $\lambda>-\alpha$ and coincides with the linearized operator $L^{\alpha + \lambda}$ from \eqref{linearized operator}.
Hence it is a Fredholm operator of index zero having an isolated eigenvalue 0.

Furthermore, there is a differentiable potential $g:\R \times D \to \R$ such that $g_u(\lambda,u)h=(G(\lambda,u),h)_{L^2}$ for all $h \in D$ in a neighborhood of $(0,0)$, given by 
$$
g(\lambda,u)=\int_\B \Bigl( \frac{1}{2}|\nabla (u+ u_{\alpha + \lambda})|^2 - \frac{|x|^{\alpha + \lambda}}{p}|u+u_{\alpha + \lambda}|^{p}\Bigr) \, dx .
$$
To apply the main theorem in \cite{kielhoefer:1988}, we need to ensure that the crossing number of the operator family $A({\lambda})$ through $\lambda =0$ is nonzero. This is a consequence of Corollary \ref{corollary-on-eigenvalue-curves}(iii), which implies that the number of negative eigenvalues of the linearized operator $L^{\alpha + \eps}=A(\eps)$ is strictly larger than that of $L^{\alpha - \eps}=A(-\eps)$ for small $\eps>0$. 
 
Therefore, \cite[Theorem, p.4]{kielhoefer:1988} implies that $(0,0)$ is a bifurcation point for the equation $G(\lambda,u)=0$, $(\lambda,u) \in \R \times D$, i.e. there exists a sequence $\left( (\lambda_n,v_n) \right)_n \subset \R \times D \setminus\{0\}$ such that 
\begin{align*}
G(\lambda_n,v_n)=0 \quad \text{for all } n, \qquad 
(\lambda_n,v_n) \to (0,0) \quad \text{in $\R \times D$ as } n \to \infty .
\end{align*}
Setting $\alpha_n:= \alpha + \lambda_n$, $u^n:=v_n + u_{\alpha_n}$ we conclude
$$
-\Delta u^n - |x|^{\alpha_n}|u^n|^{p-2} u^n 
=G(\lambda_n,v_n)=0 ,
$$
i.e. $u^n$ is a solution of \eqref{1.4}. Moreover, $u^n \to u_\alpha$ in $D$. We may therefore deduce by elliptic regularity -- using the fact that the RHS of (\ref{1.4}) is H\"older continuous in $x$ and $u$ -- that the sequence $(u^n)_n$ is bounded in $C^{2,\rho}(\ov\B)$ for some $\rho>0$, and from this we deduce that $u^n \to u_\alpha \in C^2(\ov\B)$. Since $u_{\alpha}$ is radially symmetric with precisely $K$ nodal domains, there exist $r_0 := 0<r_1 < \cdots < r_K := 1$ such that, for $i=1, \ldots, K$,  
$$
u_\alpha(x)  =0, \;(-1)^i \del_r u^n (x) >0  \quad \text{for } |x|=r_i \qquad \text{and}\qquad 
(-1)^{i-1} u_\alpha(x)  >0 \; \text{for } r_{i-1} < |x| < r_i, 
$$
where $\del_r$ denotes the derivative in the radial direction. Consequently, there exist $\eps,\delta>0$ such that, after passing to a subsequence, 
$$
(-1)^{i+1} u^n(x)  > \eps \quad \text{for $r_{i-1} + \delta < |x| < r_i - \delta$, $n \in \N$}
$$
and 
$$
(-1)^i \del_r u^n (x) >0 \quad \text{for $r_{i} - \delta < |x| < r_i + \delta$, $n \in \N$.}
$$ 
We conclude that for $i=1,\dots,K-1$ and each direction $w \in \mathbb {S}^{N-1}$ the function 
$$
(r_i - \delta, r_i + \delta) \to \R, \quad t \mapsto u^n(tw)
$$
has precisely one zero, which we denote by $r_{i,n}(w)$. In particular, the nodal domains of $u^n$ are given by 
$$ 
\Omega_1 :=\left\{ x \in \B : |x|< r_{1,n}\left(\frac{x}{|x|} \right) \right\} \quad \text{and}\quad 
\Omega_i := \left\{ x \in \B: r_{i-1,n}\left(\frac{x}{|x|}\right) <|x|< r_{i,n}\left(\frac{x}{|x|}\right) \right\} 
$$
for $i=2, \ldots K$. Consequently, $0 \in \Omega_1$, $\Omega_1$ is homeomorphic to a ball, and $\Omega_2, \ldots, \Omega_K$ are homeomorphic to annuli. Finally, we note that $u^n=v_n + u_{\alpha_n}$ is nonradial, since $v_n \not \equiv 0$ and 
$u_{\alpha_n}$ is the \emph{unique} radial solution of \eqref{1.4} with $\alpha= \alpha_n$ and with $K$ nodal domains.
\end{proof}

\end{document}